\documentclass[12pt,twoside,reqno]{amsart}

\usepackage[OT1]{fontenc}    %
\usepackage{type1cm}         %
\usepackage{amsthm}          %
\usepackage[dvips]{graphicx} 
\usepackage[bf]{caption}     
\usepackage{psfrag}          

\numberwithin{equation}{section}

\theoremstyle{plain}
\newtheorem{theorem}{Theorem}[section]
\newtheorem{corollary}[theorem]{Corollary}
\newtheorem{proposition}[theorem]{Proposition}
\newtheorem{lemma}[theorem]{Lemma}

\theoremstyle{remark}
\newtheorem{remark}[theorem]{Remark}

\newtheorem{example}[theorem]{Example}

\theoremstyle{definition}
\newtheorem{definition}[theorem]{Definition}

\newcommand{\HH}{\mathcal{H}}

\newcommand{\R}{\mathbb{R}}
\newcommand{\C}{\mathbb{C}}

\newcommand{\Z}{\mathbb{Z}}
\newcommand{\N}{\mathbb{N}}
\newcommand{\hhh}{\mathtt{h}}
\newcommand{\iii}{\mathtt{i}}
\newcommand{\jjj}{\mathtt{j}}
\newcommand{\kkk}{\mathtt{k}}
\newcommand{\eps}{\varepsilon}
\newcommand{\fii}{\varphi}
\newcommand{\roo}{\varrho}
\newcommand{\ys}{\overline{s}}
\newcommand{\as}{\underline{s}}

\DeclareMathOperator{\dimm}{dim_M}
\DeclareMathOperator{\dimh}{dim_H}

\DeclareMathOperator{\dist}{dist}
\DeclareMathOperator{\diam}{diam}
\DeclareMathOperator*{\Lim}{Lim}
\DeclareMathOperator{\proj}{proj}
\DeclareMathOperator{\conv}{conv}

\DeclareMathOperator{\por}{por}

\newenvironment{labeledlist}[1]
{ \renewcommand{\theenumi}{({#1}\arabic{enumi})}
  \renewcommand{\labelenumi}{({#1}\arabic{enumi})}
  \begin{enumerate} }
{ \end{enumerate} }

\begin{document}

\title{Separation conditions on controlled Moran constructions}

\author{Antti K\"aenm\"aki \and Markku Vilppolainen} 

\address{Department of Mathematics and Statistics \\
         P.O. Box 35 (MaD) \\
         FI-40014 University of Jyv\"askyl\"a \\
         Finland}
\email{antakae@maths.jyu.fi}
\email{vilppola@maths.jyu.fi}

\thanks{AK acknowledges the support of the Academy of Finland (project \#114821).}
\date{\today}

\begin{abstract}
  It is well known that the open set condition and the
  positivity of the $t$-dimensional Hausdorff measure are equivalent
  on self-similar sets, where $t$ is the zero of the topological
  pressure. We prove an analogous result for a class of Moran
  constructions and we study different kinds of Moran constructions
  with this respect.
\end{abstract}

\maketitle

\section{Introduction}

%
The widely studied class of self-similar sets was introduced by
Hutchinson \cite{Hutchinson1981}. A mapping $\fii \colon \R^d \to
\R^d$ is called a \emph{similitude
mapping} if there is $s > 0$ such that $|\fii(x) - \fii(y)| = s|x-y|$
whenever $x,y \in \R^d$.
If the similitude mappings $\fii_1,\ldots,\fii_k$ are
\emph{contractive}, that is, all the Lipschitz constants are strictly
less than one, then a nonempty compact set $E \subset \R^d$ is
called \emph{self-similar} provided that it satisfies
\begin{equation*}
  E = \fii_1(E) \cup \cdots \cup \fii_k(E).
\end{equation*}
From this, one can easily see that the set $E$ consists of smaller and
smaller pieces which are geometrically similar to $E$. However, the
self-similar structure is hard to recognize if these pieces overlap
too much. Hutchinson \cite{Hutchinson1981} used a separation
condition which guarantees that we can distinguish the pieces. The idea goes
back to Moran \cite{Moran1946} who studied similar constructions
but without mappings. In the \emph{open set condition},
it is required that there exists an open set $V$ such that all the images
$\fii_i(V)$ are pairwise disjoint and contained in $V$. Lalley
\cite{Lalley1988} used a stronger version of the open set
condition. In the \emph{strong open set condition},
it is required that the open set $V$ above can be chosen such that $V
\cap E \ne \emptyset$.

Assuming the open set condition, Hutchinson \cite[\S
5.3]{Hutchinson1981} proved that the $t$-dimensional Hausdorff measure
$\HH^t$ of $E$ is positive, where $t$ is the zero of the so-called
topological pressure. See also Moran \cite[Theorem III]{Moran1946} for
the corresponding theorem for the Moran constructions. Schief
\cite[Theorem 2.1]{Schief1994} showed, extending the ideas of Bandt
and Graf \cite{BandtGraf1992}, that the open set condition is not only
sufficient but also a necessary condition for the positivity of the
Hausdorff measure. In fact, he proved that $\HH^t(E)>0$ implies the
strong open set condition. Later, Peres, Rams, Simon, and Solomyak
\cite[Theorem 1.1]{PeresRamsSimonSolomyak2001} showed
that this equivalence also holds for self-conformal sets. See also
Fan and Lau \cite{FanLau1999}, Lau, Rao and Ye \cite{LauRaoYe2001},
and Ye \cite{Ye2002}.
Observe that in these results, it is essential to have finite number of
mappings. Szarek and W\c{e}drychowicz \cite{SzarekWedrychowicz2005}
have shown that in the infinite case the open set condition does not
necessarily imply the strong open set condition.

The main theme in this article is to study the relationship between
separation conditions and the Hausdorff measure on limit sets of
Moran constructions. More precisely, we study what can be said about
Schief's result in this setting. Since the open set condition requires
the use of mappings, we introduce a representative form for it to
be used on Moran constructions. We also study invariant sets of
certain iterated function systems without assuming
conformality. We generalize many classical results into these settings.

The article is organized as follows. In Section
\ref{sec:semiconformal}, we introduce the concept of semiconformal
measure on the symbol space and prove the existence of such measures.
The projection of a semiconformal measure onto the limit set of a
Moran construction gives us valuable information about the limit set
provided that the 
construction is properly controlled and the pieces used in the
construction are appropriately separated. We introduce the definitions
of the controlled Moran construction (CMC) and suitable separation
conditions in Section \ref{sec:CMC}. We also specify a class of CMC's, the so
called tractable CMC's, for which a natural separation condition is
equivalent to the positivity of $\HH^t(E)$, where $E$ is the limit set
and $t$ the zero of the topological pressure. In Section
\ref{sec:congruent}, we
consider a subclass of tractable CMC's, which we call semiconformal
CMC's. We show that this class has properties that allow us to
consider it as a natural replacement of the class of conformal
iterated function systems into the setting of Moran
constructions. Without the assumption of conformality, we study
separation conditions on iterated function systems in
Section \ref{sec:congruent_IFS}. The last section is
devoted to examples.

\section{Semiconformal measure} \label{sec:semiconformal}

In this section, we work only in the symbol space. We begin by fixing
some notation to be used throughout this article. As
usual, let $I$ be a finite set with at least two elements.
Put $I^* = \bigcup_{n=1}^\infty I^n$ and $I^\infty = I^\N$. Now for each
$\iii \in I^*$, there is $n \in \N$ such that $\iii = (i_1,\ldots,i_n)
\in I^n$. We call this $n$ the \emph{length} of $\iii$ and we
denote $|\iii|=n$. The length of elements in $I^\infty$ is
infinity. Moreover, if $\iii \in I^*$ and $\jjj \in I^* \cup I^\infty$,
then with the notation $\iii\jjj$, we mean the element obtained by
juxtaposing the terms of $\iii$ and $\jjj$. For $\iii \in I^*$ and $A
\subset I^\infty$, we define $[\iii;A] = \{ \iii\jjj : \jjj \in A \}$
and we call the set $[\iii] = [\iii;I^\infty]$ a \emph{cylinder set} of level
$|\iii|$. If $\jjj \in I^* \cup I^\infty$ and $1\le n < |\jjj|$, we
define $\jjj|_n$ to be the unique element $\iii \in I^n$ for which
$\jjj \in [\iii]$. We also denote $\iii^- = \iii|_{|\iii|-1}$.
With the
notation $\iii \bot \jjj$, we mean that the elements $\iii,\jjj \in
I^*$ are \emph{incomparable}, that is, $[\iii] \cap [\jjj] =
\emptyset$. We call a set $A \subset I^*$
incomparable if all of its elements are mutually incomparable.
Finally, with the notation $\iii \land \jjj$, we mean the common
beginning of $\iii \in I^*$ and $\jjj \in I^*$, that is, $\iii \land
\jjj = \iii|_n = \jjj|_n$, where $n = \min\{ k-1 : \iii|_k \ne \jjj|_n
\}$.

Defining
\begin{equation*}
  |\iii - \jjj| =
  \begin{cases}
    2^{-|\iii \land \jjj|}, \quad &\iii \ne \jjj \\
    0, &\iii = \jjj
  \end{cases}
\end{equation*}
for each $\iii,\jjj \in I^\infty$, the couple $(I^\infty,|\cdot|)$ is
a compact metric space. We call $(I^\infty,|\cdot|)$ a \emph{symbol
  space} and an element $\iii \in I^\infty$ a \emph{symbol}. If there
is no danger of misunderstanding, let us call also an element $\iii
\in I^*$ a symbol. Define the \emph{left
shift} $\sigma \colon I^\infty \to I^\infty$ by setting
\begin{equation*}
 \sigma(i_1,i_2,\ldots) = (i_2,i_3,\ldots).
\end{equation*}
With the notation $\sigma(i_1,\ldots,i_n)$, we mean the symbol
$(i_2,\ldots,i_n) \in I^{n-1}$. Observe that to be precise in our
definitions, we need to work with ``empty symbols'', that is, symbols
with zero length. However, this is left to the reader.

We now present
sufficient conditions for the existence of the so-called semiconformal
measure. Our presentation here has common points with \cite{Bowen1975}
and \cite[\S 2.1.2]{Barreira1996}.
Suppose the collection $\{ s_\iii > 0 : \iii \in I^* \}$
satisfies the following two assumptions:
\begin{labeledlist}{S}
  \item There exists a constant $D \ge 1$ such that
  \begin{equation*}
  D^{-1}s_\iii s_\jjj \le s_{\iii\jjj} \le Ds_\iii s_\jjj
  \end{equation*}
  whenever $\iii,\jjj \in I^*$. \label{S1}
  \item $\max_{\iii \in I^n} s_\iii \to 0$ as $n \to \infty$. \label{S2}
\end{labeledlist}
Given $t \ge 0$, we define the \emph{topological pressure} to be
\begin{equation*}
  P(t) = \lim_{n \to \infty} \tfrac{1}{n} \log \sum_{\iii \in I^n} s_\iii^t.
\end{equation*}
The limit above exists by the standard theory of subadditive sequences
since
\begin{equation*}
  \sum_{\iii \in I^{n+m}} s_\iii^t \le D^t \sum_{\iii \in I^{n+m}}
  s_{\iii|_n}^t s_{\sigma^n(\iii)}^t = D^t \sum_{\iii \in I^n} s_\iii^t \sum_{\jjj \in
  I^m} s_\jjj^t
\end{equation*}
using \ref{S1}. 

As a function, $P \colon [0,\infty) \to \R$ is convex: Let $0 \le t_1
  \le t_2$ and $\lambda \in (0,1)$. Now H\"older's inequality implies
\begin{equation*}
  \sum_{\iii \in I^n} s_\iii^{\lambda t_1 + (1-\lambda)t_2} =
  \sum_{\iii \in I^n} (s_\iii^{t_1})^\lambda(s_\iii^{t_2})^{1-\lambda}
  \le \biggl( \sum_{\iii \in I^n} s_\iii^{t_1} \biggr)^\lambda \biggl(
  \sum_{\iii \in I^n} s_\iii^{t_2} \biggr)^{1-\lambda}
\end{equation*}
from which the claim follows. According to \ref{S2}, we may choose $n
\in \N$ so that $\max_{\iii \in I^n} s_\iii < D^{-1}$. Then,
using \ref{S1}, we have
\begin{equation} \label{eq:arvio_lemmalle}
  P(t) \le \lim_{k \to \infty} \tfrac{1}{kn} \log \biggl( D^t
  \sum_{\iii \in I^n} s_\iii^t \biggl)^k 
  \le \tfrac1n \log \bigl(D\max_{\iii \in I^n} s_\iii\bigr)^t +
  \tfrac1n \log\#I^n.
\end{equation}
Hence $P(t) \to -\infty$ as $t \to \infty$ and,
noting that $P(0) = \log\#I > 0$, there exists a unique $t \ge 0$
for which $P(t)=0$.

\begin{lemma} \label{thm:summa_arvio}
  Suppose $t \ge 0$. Then
  \begin{equation*}
    D^{-t} e^{nP(t)} \le \sum_{\iii\in I^n} s_\iii^t \le D^t e^{nP(t)}
  \end{equation*}
  whenever $n \in \N$.
\end{lemma}

\begin{proof}
  Since
  \begin{equation*}
    P(t) \ge \lim_{k \to \infty} \tfrac{1}{kn} \log\biggl( D^{-t}
    \sum_{\iii \in I^n} s_\iii^t \biggr)^k = \log\biggl( \sum_{\iii
    \in I^n} s_\iii^t \biggr)^{1/n} + \log D^{-t/n}
  \end{equation*}
  by \ref{S2}, we get
  \begin{equation*}
    \sum_{\iii \in I^n} s_\iii^t \le D^te^{nP(t)}
  \end{equation*}
  for each $n \in \N$. The other inequality follows similarly from
  \eqref{eq:arvio_lemmalle}.
\end{proof}

Let $l^\infty$ be the linear space of all bounded sequences on the
real line. Recalling \cite[Theorem 7.2]{Munroe1953}, we say that the
\emph{Banach limit} is any mapping $L \colon l^\infty \to \R$ for which
\begin{labeledlist}{L}
  \item $L$ is linear, \label{L:linear}
  \item $L\bigl( (x_n)_{n \in \N} \bigr) = L\bigl( (x_{n+1})_{n \in
     \N} \bigr)$, \label{L:invariant}
  \item $\liminf_{n \to \infty} x_n \le L\bigl( (x_n)_{n \in \N}
    \bigr) \le \limsup_{n \to \infty} x_n$. \label{L:limitbounds}
\end{labeledlist}
To simplify the notation, we denote $\Lim_n x_n = L\bigl( (x_n)_{n \in
  \N} \bigr)$.

We call a Borel probability measure $\mu$ on $I^\infty$
\emph{$t$-semiconformal} if there exists a constant $c \ge 1$ such
that
\begin{equation*}
  c^{-1} e^{-|\iii|P(t)} s_\iii^t \le \mu([\iii]) \le c
  e^{-|\iii|P(t)} s_\iii^t
\end{equation*}
whenever $\iii \in I^*$.
A Borel probability measure $\mu$ on $I^\infty$ is called \emph{invariant} if
$\mu([\iii]) = \mu\bigl( \sigma^{-1}([\iii]) \bigr)$ for each $\iii
\in I^*$ and \emph{ergodic} if $\mu(A)=0$ or $\mu(A)=1$ for every
Borel set $A \subset I^\infty$ for which $A = \sigma^{-1}(A)$.
The use of the Banach
limit is a rather standard tool in producing an invariant measure from a
given measure, for example, see \cite[Corollary 1]{Waterman1975} and
\cite[Theorem 3.8]{MauldinUrbanski1996}.
In the following theorem, we construct a family of semiconformal
measures by applying the Banach limit to a suitable collection of
bounded set functions.

\begin{theorem} \label{thm:semikonforminen}
  For each $t \ge 0$ there exists a unique invariant
  $t$-semi\-conformal measure. Furthermore, it is ergodic.
\end{theorem}

\begin{proof}
  Define for each $\iii \in I^*$
  and $n \in \N$
  \begin{equation} \label{eq:massa}
    \nu_n(\iii) = \frac{\sum_{\jjj \in I^n}
    s_{\iii\jjj}^t}{\sum_{\jjj \in I^{|\iii|+n}} s_\jjj^t}.
  \end{equation}
  Letting $\nu(\iii) = \Lim_n \nu_n(\iii)$ and using
  using \ref{L:linear} and \ref{L:invariant}, we have
  \begin{equation} \label{eq:ysi}
  \begin{split}
    \sum_{j \in I} \nu(\iii j) &= \sum_{j \in I} \Lim_n \nu_n(\iii j)
    = \Lim_n \frac{\sum_{j \in I}\sum_{\jjj \in I^n}
    s_{\iii j\jjj}^t}{\sum_{\jjj \in I^{|\iii| + 1 + n}} s_\jjj^t} \\
    &= \Lim_n \nu_{n+1}(\iii) = \Lim_n \nu_n(\iii) = \nu(\iii)
  \end{split}
  \end{equation}
  whenever $\iii \in I^*$. Since, by Lemma \ref{thm:summa_arvio},
  \begin{align*}
    \nu_n(\iii) &\le D^t e^{-(|\iii|+n)P(t)} \sum_{\jjj \in I^n}
    s_{\iii\jjj}^t \le D^{2t} e^{-(|\iii|+n)P(t)} s_\iii^t \sum_{\jjj
      \in I^n} s_\jjj^t \\ &\le D^{3t} e^{-|\iii|P(t)} s_\iii^t
  \end{align*}
  and similarly the other way around, we have, using
  \ref{L:limitbounds},
  \begin{equation} \label{eq:kymppi}
    D^{-3t} e^{-|\iii|P(t)} s_\iii^t \le \nu(\iii) \le D^{3t}
    e^{-|\iii|P(t)} s_\iii^t.
  \end{equation}

  Next define for each $\iii \in I^*$ and $n \in \N$
  \begin{equation*}
    \mu_n(\iii) = \sum_{\jjj \in I^n} \nu(\jjj\iii).
  \end{equation*}
  Letting $\mu(\iii) = \Lim_n \mu_n(\iii)$, we have $\mu(\iii) > 0$
  and, using \eqref{eq:ysi},
  \begin{equation} \label{eq:yytoo}
    \sum_{j \in I} \mu(\iii j) = \Lim_n \sum_{j \in I} \sum_{\jjj \in
      I^n} \nu(\jjj\iii j) = \mu(\iii)
  \end{equation}
  whenever $\iii \in I^*$. Observe also that
  \begin{equation} \label{eq:kaatoo}
    \sum_{j \in I} \mu(j\iii) = \Lim_n \sum_{j \in I} \sum_{\jjj \in
      I^n} \nu(\jjj j\iii) = \Lim_n \mu_{n+1}(\iii) = \mu(\iii)
  \end{equation}
  whenever $\iii \in I^*$. Using now \eqref{eq:kymppi} and Lemma
  \ref{thm:summa_arvio}, we have
  \begin{align*}
    \mu_n(\iii) &\le D^{3t} \sum_{\jjj \in I^n} e^{-|\iii\jjj|P(t)}
    s_{\iii\jjj}^t \le D^{4t} e^{-|\iii|P(t)} s_\iii^t \sum_{\jjj \in
      I^n} e^{-|\jjj|P(t)} s_\jjj^t \\ &\le D^{5t} e^{-|\iii|P(t)} s_\iii^t
  \end{align*}
  and similarly the other way around. Hence
  \begin{equation} \label{eq:kootoo}
    D^{-5t} e^{-|\iii|P(t)} s_\iii^t \le \mu(\iii) \le D^{5t}
    e^{-|\iii|P(t)} s_\iii^t.
  \end{equation}

  Now, identifying $\iii \in I^*$ with the
  cylinder $[\iii]$, we notice, using \eqref{eq:yytoo}, that $\mu$ is
  a probability measure on the semi-algebra of all cylinder
  sets. Hence, using the Carath\'eodory-Hahn Theorem (see
  \cite[Theorem 11.20]{WheedenZygmund1977}), $\mu$ extends to a Borel
  probability measure on $I^\infty$. Observe that by \eqref{eq:kaatoo}
  and \eqref{eq:kootoo}, $\mu$ is an invariant $t$-semiconformal measure.


  We shall next prove that $\mu$ is ergodic. We have learned the
  following argument from the proof of \cite[Theorem
  3.8]{MauldinUrbanski1996}. Assume on the contrary that there exists a
  $\mu$-measurable set $A \subset I^\infty$ such that $\sigma^{-1}(A)
  = A$ and $0 < \mu(A) < 1$. Fix $\iii \in I^*$ and take an
  incomparable set $R \subset I^*$ for which $I^\infty \setminus A
  \subset \bigcup_{\jjj \in R} [\jjj]$ and
  \begin{equation*}
    \sum_{\jjj \in R} \mu([\iii\jjj]) \le 2\mu([\iii;I^\infty
    \setminus A]).
  \end{equation*}
  Using \eqref{eq:kootoo}, we infer
  \begin{align*}
    \mu([\iii;I^\infty \setminus A]) &\ge 2^{-1}D^{-6t} s_\iii^t
    \sum_{\jjj \in R} e^{-|\iii\jjj|P(t)} s_\jjj^t \\ &\ge
    2^{-1}D^{-16t} e^{|\iii|P(t)} \mu([\iii]) \sum_{\jjj \in R}
    e^{-|\iii\jjj|P(t)} e^{|\jjj|P(t)} \mu([\jjj]) \\
    &\ge 2^{-1}D^{-16t} \mu([\iii])\mu(I^\infty\setminus A).
  \end{align*}
  Therefore
  \begin{equation} \label{eq:koltoista}
  \begin{split}
    \mu\bigl( \sigma^{-n}(A) \cap [\iii] \bigr) &= \mu([\iii;A]) =
    \mu([\iii]) - \mu([\iii;I^\infty \setminus A]) \\
    &\le \bigl( 1 - 2^{-1}D^{-16t} \mu(I^\infty \setminus A) \bigr)
    \mu([\iii])
  \end{split}
  \end{equation}
  for each $\iii \in I^*$.
  Denote $\gamma = \bigl( 1 - 2^{-1}D^{-16t} \mu(I^\infty \setminus A)
  \bigr)$ and $\eta = (1 + \gamma^{-1})/2$. Take an incomparable set $R
  \subset I^*$ for which $A \subset \bigcup_{\iii \in R} [\iii]$ and
  $\sum_{\iii \in R} \mu([\iii]) \le \eta\mu(A)$. Since now, using
  \eqref{eq:koltoista},
  \begin{align*}
    \mu(A) &= \sum_{\iii \in R} \mu(A \cap [\iii]) = \sum_{\iii \in R}
    \mu\bigl( \sigma^{-n}(A) \cap [\iii] \bigr) \\
    &\le \sum_{\iii \in R} \gamma\mu([\iii]) \le \gamma\eta\mu(A) < \mu(A),
  \end{align*}
  we have finished the proof of the ergodicity.

  To prove the uniqueness, assume that $\tilde \mu$ is another invariant
  $t$-semiconformal measure. Now there exists $c \ge 1$ such that
  $\tilde\mu([\iii]) \le c \mu([\iii])$ whenever $\iii \in
  I^*$. According to the uniqueness of the Carath\'eodory-Hahn
  extension, this inequality implies that also $\tilde\mu \le
  c\mu$. Using the ergodicity of the measure $\mu$, it follows that
  $\tilde \mu = \mu$, see \cite[Theorem 6.10]{Walters1982}. The proof
  is finished.
\end{proof}


Let us next prove two lemmas for future reference. Define for $\iii
\in I^*$
\begin{equation*}
  \Omega_\iii = \{ \jjj \in I^\infty : \sigma^{n-1}(\jjj) \in [\iii]
                   \text{ with infinitely many } n \in \N \}
\end{equation*}
and
\begin{equation*}
  \Omega_\iii^0 = \{ \jjj \in I^\infty : \sigma^{n-1}(\jjj) \notin
  [\iii] \text{ for every } n \in \N \}.
\end{equation*}

\begin{lemma} \label{thm:aareton_toisto}
  Suppose $\mu$ is an invariant ergodic Borel probability
  measure on $I^\infty$. Then $\mu(\Omega_\iii^0)=0$ and
  $\mu(\Omega_\iii)=1$ for every $\iii \in I^*$ provided that
  $\mu([\iii])>0$.
\end{lemma}

\begin{proof}
  Take $\iii \in I^*$ such that $\mu([\iii])>0$. Notice that
  $\sigma^{-1}(I^\infty \setminus \Omega_\iii^0) \subset I^\infty
  \setminus \Omega_\iii^0$ and due to the invariance of $\mu$, it holds
  that $\mu\bigl(\sigma^{-1}(I^\infty \setminus \Omega_\iii^0)\bigr) =
  \mu(I^\infty \setminus \Omega_\iii^0)$. Since $\Omega_\iii =
  \bigcap_{n=0}^\infty \sigma^{-n}(I^\infty \setminus \Omega_\iii^0)$, we
  have $\sigma^{-1}(\Omega_\iii) = \Omega_\iii$ and using the
  ergodicity of $\mu$, we have either $\mu(\Omega_\iii)=0$ or
  $\mu(\Omega_\iii)=1$. Since
  \begin{equation*}
    \mu(\Omega_\iii) = \lim_{n \to \infty} \mu\bigl( \sigma^{-n}(I^\infty
    \setminus \Omega_\iii^0) \bigr) = \mu(I^\infty \setminus
    \Omega_\iii^0) \ge \mu([\iii]) > 0,
  \end{equation*}
  it follows that $\mu(I^\infty \setminus \Omega_\iii^0) =
  \mu(\Omega_\iii) = 1$. The proof is finished.
\end{proof}

Assume that $I$ has at least three elements. For a fixed $j \in
I$, we denote $I_j = I \setminus \{ j \}$ and define
\begin{equation*}
  P_j(t) = \lim_{n \to \infty} \tfrac{1}{n} \log \sum_{\iii \in I_j^n}
  s_\iii^t.
\end{equation*}

\begin{lemma} \label{thm:Pj_pieni}
  Suppose $P(t)=0$. Then $P_j(t)<0$ for every $j \in I$.
\end{lemma}

\begin{proof}
  Using Theorem \ref{thm:semikonforminen}, we denote with $\mu$ the
  invariant ergodic Borel probability measure on $I^\infty$ for which
  \begin{equation*}
    c^{-1}s_\iii^t \le \mu([\iii]) \le cs_\iii^t
  \end{equation*}
  for a constant $c \ge 1$ whenever $\iii \in I^*$.
  Assume now on the contrary that there is $j \in I$ such that
  $P_j(t)=0$. Using Theorem \ref{thm:semikonforminen}, we denote with
  $\mu_j$ the unique invariant $t$-semiconformal measure on
  $I_j^\infty$. Observe that there exists a constant $c_j \ge 1$ such
  that
  \begin{equation*}
    c_j^{-1}s_\iii^t \le \mu_j([\iii]) \le c_js_\iii^t
  \end{equation*}
  whenever $\iii \in I_j^*$. Notice also that $\mu_j(I^\infty
  \setminus I_j^\infty) = 0$ and $\mu(I_j^\infty) = 0$ by Lemma
  \ref{thm:aareton_toisto}. Defining $\lambda_j = \tfrac12(\mu +
  \mu_j)$, we have for each $\iii \in I_j^*$
  \begin{align*}
    \lambda_j([\iii]) &= \lambda_j([\iii] \setminus I_j^\infty) +
    \lambda_j([\iii] \cap I_j^\infty) \\
    &= \tfrac12\mu([\iii]) + \tfrac12\mu_j([\iii])
    \le \tfrac12 (c + c_j)s_\iii^t
  \end{align*}
  and similarly the other way around. Hence also $\lambda_j$ is
  invariant and $t$-semi\-conformal on $I_j^\infty$. From the
  uniqueness, we infer $\lambda_j = \mu_j$, and therefore
  \begin{equation*}
    1 = \mu_j(I_j^\infty) = \lambda_j(I_j^\infty) = \tfrac12(\mu +
    \mu_j)(I_j^\infty) = \tfrac12.
  \end{equation*}
  This contradiction finishes the proof.
\end{proof}

\section{Controlled Moran construction} \label{sec:CMC}

The collection $\{ X_\iii \subset \R^d : \iii \in I^* \}$
of compact sets with positive diameter is called a \emph{controlled
  Moran construction (CMC)} if
\begin{labeledlist}{M}
  \item $X_{\iii i} \subset X_\iii$ as $\iii \in I^*$ and $i \in I$,
  \label{M1}
  \item there exists a constant $D \ge 1$ such that
    \begin{equation*}
      D^{-1} \le
      \frac{\diam(X_{\iii\jjj})}{\diam(X_\iii)\diam(X_\jjj)}
      \le D
    \end{equation*}
  whenever $\iii,\jjj \in I^*$, \label{M2}
  \item there exists $n \in \N$ such that
    \begin{equation*}
      \max_{\iii \in I^n} \diam(X_\iii) < D^{-1}.
    \end{equation*}
  \label{M3}
\end{labeledlist}

\begin{lemma} \label{thm:rajanolla}
  Given CMC, there are constants $c>0$ and $0<\roo < 1$ such that
  $\max_{\iii \in I^n} \diam(X_\iii) \le c\roo^n$ for all $n \in \N$.
\end{lemma}

\begin{proof}
  Using \ref{M3}, we find $k \in \N$ and $0<a<1$ such that
  $\diam(X_\iii) < a/D$ for every $\iii \in I^k$. Fix $n > k$, take
  $\iii \in I^n$ and denote $\iii = \iii_1\iii_2\cdots\iii_l$, where
  $l-1$ is the integer part of $n/k$, $\iii_j \in I^k$ for $j \in \{
  1,\ldots,l-1 \}$, and $0<|\iii_l|\le k$. Since now, by \ref{M2},
  \begin{align*}
    \diam(X_\iii) &\le D^{l-1}\diam(X_{\iii_1})\diam(X_{\iii_2}) \cdots
    \diam(X_{\iii_{l-1}})\diam(X_{\iii_l}) \\
    &\le D^{l-1}\bigl(a/D\bigr)^{l-1} \max_{0<|\iii|\le k} \diam(X_\iii)
    \le a^{-1} \max_{0<|\iii|\le k} \diam(X_\iii) \bigl(a^{1/k}\bigr)^n,
  \end{align*}
  the proof is finished.
\end{proof}

Using the assumption \ref{M1} and Lemma \ref{thm:rajanolla}, we define a
\emph{projection mapping} $\pi\colon I^\infty \to X$ such that
\begin{equation*}
  \{ \pi(\iii) \} = \bigcap_{n=1}^\infty X_{\iii|_n}
\end{equation*}
as $\iii \in I^\infty$. It is clear that $\pi$ is continuous. The
compact set $E = \pi(I^\infty)$ is called the \emph{limit set} (of the CMC).
We call a Borel probability
measure $m$ on $E$ \emph{$t$-semiconformal} if
there exists a constant $c \ge 1$ such that
\begin{equation*}
  c^{-1}\diam(X_\iii)^t \le m(X_\iii) \le c\diam(X_\iii)^t
\end{equation*}
whenever $\iii \in I^*$ and
\begin{equation*}
  m(X_\iii \cap X_\jjj) = 0
\end{equation*}
whenever $\iii \bot \jjj$.
Observe that in Section \ref{sec:semiconformal} we defined a
semiconformal measure on $I^\infty$. The overlapping terminology
should not be confusing as it is clear from the content
which definition we use. Furthermore, for each $t \ge 0$, we set
\begin{equation} \label{eq:topo2}
  P(t) = \lim_{n \to \infty} \tfrac1n \log\sum_{\iii \in I^n}
  \diam(X_\iii)^t
\end{equation}
provided that the limit exists. It follows straight from the
definition that if there exists a $t$-semiconformal measure
on $E$ then $P(t) = 0$. Recalling Lemma \ref{thm:summa_arvio},
the equation $P(t)=0$ gives a natural upper bound for the Hausdorff
dimension of $E$, $\dimh(E) \le t$.

The following proposition provides sufficient conditions for the
existence of the $t$-semiconformal measure on $E$.
We say that a CMC satisfies a \emph{bounded overlapping property}
if $\sup_{x \in E} \sup\bigl\{ \# R : R \subset \{ \iii \in I^* : x \in
X_\iii \} \text{ is incomparable} \bigr\} < \infty$.
Observe that in the proposition the assumption according to which for
each $\iii,\jjj \in I^*$ and $\hhh \in I^\infty$ it holds that
$\pi(\iii\hhh) \in X_{\iii\jjj}$ whenever $\pi(\hhh) \in X_\jjj$ is
essential, see Example \ref{ex:nontractable}.

\begin{proposition} \label{thm:semileikkaus}
  Given a CMC, the limit in \eqref{eq:topo2} exists and there is a
  unique $t \ge 0$ such that $P(t)=0$. Assuming $P(t)=0$, there exists
  an invariant ergodic Borel probability
  measure $\mu$ on $I^\infty$ and constants $c,c'>0$ such that
  \begin{equation*}
    c^{-1}\diam(X_\iii)^t \le \mu([\iii]) \le c\diam(X_\iii)^t
  \end{equation*}
  whenever $\iii \in I^*$. Denoting $m = \mu \circ \pi^{-1}$, we
  have $\HH^t(A) \le c'm(A)$ for every $m$-measurable $A \subset
  E$. Furthermore, if in addition the CMC satisfies the bounded
  overlapping property and for each $\iii,\jjj \in I^*$ and $\hhh \in
  I^\infty$ it holds that $\pi(\iii\hhh) \in X_{\iii\jjj}$ whenever
  $\pi(\hhh) \in X_\jjj$ then $m$ is a $t$-semiconformal measure on $E$.
\end{proposition}

\begin{proof}
  According to \ref{M2} and Lemma \ref{thm:rajanolla}, the collection
  $\{ \diam(X_\iii) : \iii \in I^* \}$ satisfies \ref{S1} and
  \ref{S2}. The proof of the first claim is now trivial. Suppose
  $P(t)=0$ and denote with $\mu$ the $t$-semiconformal measure on
  $I^\infty$ associated to this collection, see Theorem
  \ref{thm:semikonforminen}. For fixed $x \in E$ and $r > 0$ take
  $\iii = (i_1,i_2,\ldots) \in I^\infty$ such that $\pi(\iii) = x$ and
  choose $n$ to be the smallest integer for which $X_{\iii|_n} \subset
  B(x,r)$. Denoting $m = \mu \circ \pi^{-1}$ and using \ref{M2}, we
  obtain
  \begin{align*}
    m\bigl( B(x,r) \bigr) &\ge m(X_{\iii|_n}) \ge \mu([\iii|_n])
    \ge c^{-1}\diam(X_{\iii|_n})^t \\
    &\ge c^{-1}D^{-t}\diam(X_{\iii|_{n-1}})^t \diam(X_{i_n})^t \\
    &\ge c^{-1}D^{-t} \min_{i \in I} \diam(X_i)^t r^t,
  \end{align*}
  which, according to \cite[Proposition 2.2(b)]{Falconer1997}, gives the
  second claim. Here with the notation $B(x,r)$, we mean the open ball
  centered at $x$ with radius $r$. Furthermore, if the bounded
  overlapping property is satisfied then the proof of \cite[Theorem
  3.7]{Kaenmaki2004} shows that
  \begin{equation*}
    m(X_\iii \cap X_\jjj) = 0
  \end{equation*}
  whenever $\iii \bot \jjj$ provided that for each $\iii,\hhh,\kkk \in
  I^*$ it holds $\mu\bigl( [\iii; \pi^{-1}(X_\hhh \cap X_\kkk)] \bigr)
  \le m(X_{\iii\hhh} \cap X_{\iii\kkk})$. This is guaranteed by our
  extra assumption. Hence
  \begin{align*}
    m(X_\iii) &= m\Bigl( X_\iii \setminus \bigcup_{\iii \bot \jjj}
    X_\jjj \cap X_\iii \Bigr) \\
    &= \mu\Bigl( \pi^{-1}(X_\iii) \setminus \bigcup_{\iii \bot \jjj}
    \pi^{-1}(X_\jjj \cap X_\iii) \Bigr) = \mu([\iii]),
  \end{align*}
  which finishes the proof of the last claim.
\end{proof}

In the definition that follows, we introduce a natural separation
condition to be used on Moran constructions. Given CMC, define for
$r>0$
\begin{equation*}
  Z(r) = \{ \iii \in I^* : \diam(X_\iii) \le r < \diam(X_{\iii^-}) \}
\end{equation*}
and if in addition $x \in E$, we set
\begin{equation*}
  Z(x,r) = \{ \iii \in Z(r) : X_\iii \cap B(x,r) \ne \emptyset \}.
\end{equation*}
It is often useful to know the cardinality of the set $Z(x,r)$.
We say that a CMC satisfies a \emph{finite clustering
property} if $\sup_{x \in E} \limsup_{r \downarrow 0} \#
Z(x,r) < \infty$.
Furthermore, if $\sup_{x \in E}\sup_{r>0}\# Z(x,r) < \infty$ then the
CMC is said to satisfy a \emph{uniform finite clustering property}.

\begin{definition} \label{def:ball_cond}
  We say that a CMC satisfies a \emph{ball condition} if there
  exists a constant $0<\delta<1$ such that for each $x \in E$ there is
  $r_0>0$ such that for every $0<r<r_0$ there exists a set $\{ x_\iii
  \in \conv(X_\iii) : \iii \in Z(x,r) \}$ such that the collection 
  $\{ B(x_\iii,\delta r) : \iii \in Z(x,r) \}$ is
  disjoint.
  If $r_0>0$ above can be chosen to be infinity for every $x \in E$ then
  the CMC is said to satisfy a \emph{uniform ball condition}.
  Here with the notation $\conv(A)$, we mean the convex hull of
  a given set $A$.
\end{definition}

We shall next prove that the (uniform) ball condition and the
(uniform) finite clustering property are equivalent.

\begin{lemma} \label{thm:yhtenainen}
  Given a compact and connected set $A \subset \R^n$ and $k \in \N$,
  there exists points $x_1,\ldots,x_k \in A$ such that the collection
  of balls $\bigl\{ B\bigl( x_i,(2k)^{-1}\diam(A) \bigr) : i \in \{
  1,\ldots,k \} \bigr\}$ is disjoint and $\# \bigl\{ i \in \{
  1,\ldots,k \} : B\bigl( x_i,(2k)^{-1}\diam(A) \bigr) \cap B\bigl(
  x,(2k)^{-1}\diam(A) \bigr) \ne \emptyset \bigr\} \le 2$ for every
  $x \in \R^n$.
\end{lemma}

\begin{proof}
  Choose $y_1,y_k \in A$ such that $|y_1 - y_k| = \diam(A)$. Denote
  the line going through $y_1$ and $y_k$ with $L$ and define for each
  $i \in \{ 2,\ldots,k-1 \}$ a point $y_i = \bigl( 1-\tfrac{i}{k}
  \bigr)y_1 + \tfrac{i}{k}y_k \in L$. Using the connectedness of $A$,
  we find for each $i \in \{ 1,\ldots,k \}$ a point $x_i \in A$ for
  which the inner product $(x_i - y_i) \cdot (y_k - y_1) = 0$. The
  proof is finished.
\end{proof}

\begin{theorem} \label{thm:bounded_separation}
  A CMC satisfies the (uniform) ball condition exactly when
  it satisfies the (uniform) finite clustering property.
\end{theorem}

\begin{proof}
  We shall prove the non-uniform case. The uniform case follows
  similarly. Assuming the ball condition,
  take $x \in E$ and $0<r<r_0$. Choose for each $\iii \in
  Z(x,r)$ a point $x_\iii \in \conv(X_\iii)$ such that the balls
  $B(x_\iii,\delta r)$ are disjoint as $\iii \in Z(x,r)$. Now clearly
  \begin{equation*}
    B(x_\iii,\delta r) \subset B\bigl( x,(2+\delta)r \bigr)
  \end{equation*}
  for every $\iii \in Z(x,r)$. Hence
  \begin{align*}
    \# Z(x,r)\delta^d r^d \alpha(d) &= \sum_{\iii \in Z(x,r)}
    \HH^d\bigl( B(x_\iii,\delta r) \bigr) = \HH^d\biggl( \bigcup_{\iii
    \in Z(x,r)} B(x_\iii,\delta r) \biggr) \\
    &\le \HH^d\bigl( B\bigl( x,(2+\delta)r \bigl) \bigr) =
    (2+\delta)^d r^d \alpha(d),
  \end{align*}
  where $\alpha(d)$ denotes the $d$-dimensional Hausdorff measure of
  the unit ball. This shows that the CMC satisfies the finite clustering
  property.

  Conversely,
  by the finite clustering property, there exists $M > 0$ such that for each
  $x \in E$ there is $r_0 > 0$ such that $\# Z(x,r) < M$ whenever
  $0<r<r_0$. Choose $\delta = (4MD)^{-1}\min_{i \in I}\diam(X_i)$ and
  for fixed $x \in E$ and $0<r<r_0$ denote the symbols of $Z(x,r)$
  with $\iii_1,\ldots,\iii_n$, where $n = \# Z(x,r)$. We shall define
  the points $x_{\iii_1},\ldots,x_{\iii_n}$ needed in the ball
  condition inductively. Choose $x_{\iii_1}$ to be any point of
  $\conv(X_{\iii_1})$ and assume the points 
  $x_{\iii_1},\ldots,x_{\iii_k}$, where $k \in \{ 1,\ldots,n-1 \}$,
  have already been chosen such that the collection of balls $\bigl\{
  B(x_{\iii_i},\delta r) : i \in \{ 1,\ldots,k \} \bigr\}$ is
  disjoint. Using Lemma \ref{thm:yhtenainen}, we find points
  $y_1,\ldots,y_{2n} \in \conv(X_{\iii_{k+1}})$ such that the collection
  $\bigl\{ B\bigl( y_j,(4n)^{-1}\diam(X_{\iii_{k+1}}) \bigr) : j \in
  \{ 1,\ldots,2n \} \bigr\}$ is disjoint. Since, using \ref{M2},
  \begin{equation*}
    \delta r \le (4MD)^{-1}\min_{i \in I}\diam(X_i)\diam(X_{\iii^-})
    \le (4n)^{-1}\diam(X_\iii)
  \end{equation*}
  for every $\iii \in Z(x,r)$, Lemma \ref{thm:yhtenainen} says also
  that the balls $B(x_{\iii_i},\delta r)$, $i \in \{ 1,\ldots,k \}$,
  can intersect at most $2k$ of balls $B\bigl(
  y_j,(4n)^{-1}\diam(X_{\iii_{k+1}}) \bigr)$, $j \in
  \{1,\ldots,2n\}$. Hence, choosing $x_{\iii_{k+1}} \in \{
  y_1,\ldots,y_{2n} \}$ such that $B\bigl(
  x_{\iii_{k+1}},(4n)^{-1}\diam(X_{\iii_{k+1}}) \bigr) \cap
  B(x_{\iii_i},\delta r) = \emptyset$ for every $i \in \{ 1,\ldots,k
  \}$, we have finished the proof.
\end{proof}

It is evident that the bounded overlapping property does not imply the finite
clustering property and in Example \ref{ex:ekaesim}, we show that the
converse does not hold either. The natural condition according to
which $\sup_{x \in E, r>0} \sup\bigl\{ \# R : R 
\subset \{ \iii \in I^* : X_\iii \cap B(x,r) \ne \emptyset \; \mbox{and}
\; \diam(X_{\iii^-}) > r \} \; \mbox{is incomparable} \bigr\} < \infty$
clearly implies both the bounded overlapping property and the uniform
finite clustering property. See also \cite[Lemma
2.7]{MauldinUrbanski1996}. However, we do not need this condition as
under a minor technical assumption, the finite clustering property
implies the bounded overlapping property.

\begin{lemma} \label{thm:technical_lemma}
  If a CMC satisfies the finite clustering property then it satisfies
  the bounded overlapping property provided that
  \begin{equation*}
    X_\iii \cap E = \pi([\iii])
  \end{equation*}
  for each $\iii \in I^*$.
\end{lemma}

\begin{proof}
  Set $M = \sup_{x \in E} \limsup_{r \downarrow 0} \# Z(x,r)$. Fix $x
  \in E$ and assume that $R \subset I^*$ is a finite and incomparable
  set such that $x \in X_\iii$ for each $\iii \in R$. Choose $r>0$
  small enough so that $\# Z(x,r) \le M$ and
  \begin{equation*}
    \min_{\jjj \in Z(x,r)}|\jjj| > \max_{\iii \in R}|\iii|.
  \end{equation*}
  According to the assumption, $x \in \bigcap_{\iii \in R}
  \pi([\iii])$, and hence, for each $\iii \in R$ there exists at least one
  $\iii^* \in Z(x,r)$ such that $\iii^*|_n = \iii$ for some $n \in
  \N$. The incomparability of $R$ now implies that $\iii^* \ne \jjj^*$ for
  distinct $\iii,\jjj \in R$. Consequently, $\# R \le \# Z(x,r) \le M$.
\end{proof}

Let us examine how the Hausdorff measure is related to the
ball condition. Bear in mind that the finite clustering property and
the ball condition are equivalent.

\begin{theorem} \label{thm:regular}
  If a CMC satisfies the uniform finite clustering property,
  $P(t)=0$, and $m$ is the measure of Proposition
  \ref{thm:semileikkaus} then there exist constants $r_0>0$ and $K\ge
  1$ such that 
  \begin{equation*}
    K^{-1}r^t \le m\bigl( B(x,r) \bigr) \le Kr^t
  \end{equation*}
  whenever $x \in E$ and $0<r<r_0$. Consequently, $\dimh(E) = \dimm(E)
  = t$.
\end{theorem}

\begin{proof}
  Suppose $P(t)=0$ and $m = \mu \circ \pi^{-1}$ is the measure of
  Proposition \ref{thm:semileikkaus}. Seeing that $\pi^{-1}\bigl(
  B(x,r) \bigr) \subset \bigcup_{\iii \in Z(x,r)} [\iii]$, we get
  for fixed $x \in E$ and $r>0$
  \begin{align*}
    m\bigl( B(x,r) \bigr) &\le \mu\biggl( \bigcup_{\iii \in Z(x,r)}
    [\iii] \biggr) \le \sum_{\iii \in Z(x,r)} \mu([\iii]) \\
    &\le c\sum_{\iii \in Z(x,r)} \diam(X_\iii)^t \le \# Z(x,r) cr^t,
  \end{align*}
  which, together with the uniform finite clustering property and the
  proof of Proposition \ref{thm:semileikkaus}, gives the first claim.

  The second claim follows immediately from \cite[Theorem
  5.7]{Mattila1995}.
\end{proof}

\begin{remark} \label{rem:bounded_hausdorff}
  We remark that in Theorem \ref{thm:regular}, the measure $m$ can be
  replaced with the Hausdorff measure $\HH^t|_E$ by recalling
  \cite[Proposition 2.2]{Falconer1997}. In fact, it is sufficient to
  assume the finite clustering property instead of the uniform finite
  clustering property to see that $\HH^t|_E$ is proportional to
  $m$. Especially, under this assumption, it holds that
  $0<\HH^t(E)<\infty$.
\end{remark}

One could easily prove that if $\HH^t|_E$ is $t$-semiconformal
for some $t \ge 0$ then there exists a set $A 
\subset E$ with $\HH^t(E \setminus A) = 0$ such that $\sup_{x \in
A}\limsup_{r \downarrow 0} \# Z(x,r) < \infty$. Since this hardly
generalizes to the whole set $E$ without any additional assumption, we
propose the following definition. We say that a CMC is
\emph{tractable} if there exists a constant $C\ge 1$ such that for
each $r>0$ we have
\begin{equation} \label{eq:tractable}
  \dist(X_{\hhh\iii},X_{\hhh\jjj}) \le C\diam(X_\hhh)r
\end{equation}
whenever $\hhh \in I^*$, $\iii,\jjj \in Z(r)$, and $\dist(X_\iii,X_\jjj) \le r$.
See Example \ref{ex:nontractable} for an example of a nontractable CMC.
Comparing the following theorem to \cite[Theorem 2.1]{Schief1994} and
\cite[Theorem 1.1]{PeresRamsSimonSolomyak2001}, we see that the
uniform ball condition is a proper substitute for the open set
condition in the setting of tractable CMC's. See also Example
\ref{ex:suorakaiteet3}.


\begin{theorem} \label{thm:hausdorff_bounded}
  A tractable CMC satisfies the uniform finite clustering property
  provided that $\HH^t(E)>0$ for the unique $t \ge 0$ satisfying $P(t)=0$.
\end{theorem}

\begin{proof}
  Assume on the contrary that for each $N \in \N$ there are $x'_N \in
  E$ and $r'_N>0$ such that $\# Z(x'_N,r'_N) \ge N$. For fixed $N \in \N$
  choose $\iii \in Z(x'_N,r'_N)$ so that $x'_N = \pi(\iii\kkk_0)$ for
  some $\kkk_0 \in I^\infty$. We define
  \begin{equation*}
    \Omega_\iii = \{ \kkk \in I^\infty : \sigma^{n-1}(\kkk) \in [\iii]
    \text{ with infinitely many } n \in \N \}
  \end{equation*}
  and taking arbitrary $\kkk \in \Omega_\iii$ and $n \in \N$ for which
  $\sigma^n(\kkk) \in [\iii]$, we denote $x = \pi(\kkk)$ and $\hhh =
  \kkk|_n$. Finally, pick $\jjj_1,\ldots,\jjj_N \in Z(x'_N,r'_N)$ such
  that $\jjj_i \ne \jjj_j$ as $i \ne j$. Since now
  $\dist(X_\iii,X_{\jjj_i}) \le r'_N$ for every $i \in \{ 1,\ldots,N
  \}$, we have, according to the assumption, that
  $\dist(X_{\hhh\iii},X_{\hhh\jjj_i}) \le C\diam(X_\hhh)r'_N$. Hence
  \begin{align*}
    \pi([\hhh\jjj_i]) &\subset X_{\hhh\jjj_i} \subset B\bigl(
    x,\diam(X_{\hhh\iii}) + \dist(X_{\hhh\iii},X_{\hhh\jjj_i}) +
    \diam(X_{\hhh\jjj_i}) \bigr) \\
    &\subset B\bigl( x,(2D+C)\diam(X_\hhh)r'_N \bigr)
  \end{align*}
  for each $i \in \{ 1,\ldots,N \}$ recalling that $x \in
  X_{\hhh\iii}$. Therefore
  \begin{equation*}
    \pi\biggl( \bigcup_{i=1}^N [\hhh\jjj_i] \biggr) \subset B(x,r_n),
  \end{equation*}
  where $r_n = (2D+C)\diam(X_{\kkk|_n})r'_N$, and
  \begin{align*}
    \frac{m\bigl( B(x,r_n) \bigr)}{r_n^t} &\ge \frac{\sum_{i=1}^N
    \mu([\hhh\jjj_i])}{r_n^t} \ge \frac{c^{-1} \sum_{i=1}^N
    \diam(X_{\hhh\jjj_i})^t}{r_n^t} \\
    &\ge \frac{c^{-1}D^{-t} \diam(X_\hhh)^t \sum_{i=1}^N
    \diam(X_{\jjj_i})^t}{(2D+C)^t\diam(X_\hhh)^t r'_N}
    \ge C_0N,
  \end{align*}
  where $\mu$ is the measure of Proposition \ref{thm:semileikkaus}, $m
  = \mu \circ \pi^{-1}$, and the constant $C_0>0$ does not depend on
  $n$ or $N$. Since $r_n \downarrow 0$ as $n \to \infty$, we obtain
  \begin{equation*}
    \limsup_{r \downarrow 0} \frac{m\bigl( B(x,r) \bigr)}{r^t} \ge C_0N
  \end{equation*}
  for all $x \in \pi(\Omega_\iii)$, which, according to
  \cite[Proposition 2.2(b)]{Falconer1997}, gives
  \begin{equation} \label{eq:hh_arvio}
    \HH^t\bigl( \pi(\Omega_\iii) \bigr) \le 2^tC_0^{-1}N^{-1} m\bigl(
    \pi(\Omega_\iii) \bigr).
  \end{equation}
  Since $1 = \mu(\Omega_\iii) \le m\bigl( \pi(\Omega_\iii) \bigr) \le
  1$ by Lemma \ref{thm:aareton_toisto}, we have, using
  \eqref{eq:hh_arvio} and Proposition \ref{thm:semileikkaus},
  \begin{align*}
    \HH^t(E) &\le \HH^t\bigl( \pi(\Omega_\iii) \bigr) + \HH^t\bigl( E
    \setminus \pi(\Omega_\iii) \bigr) \\
    &\le 2^tC_0^{-1}N^{-1} m\bigl( \pi(\Omega_\iii) \bigr) + c'm\bigl( E
    \setminus \pi(\Omega_\iii) \bigr) \le 2^tC_0^{-1}N^{-1},
  \end{align*}
  which leads to a contradiction as $N \to \infty$.
\end{proof}

To summarize the implications of the previous theorem, we finish this
section with the following corollary.

\begin{corollary} \label{thm:tractable_corollary}
  For a tractable CMC, the following are equivalent:
  \begin{enumerate}
  \item The ball condition.
  \item The uniform ball condition.
  \item $\HH^t(E) > 0$, where $P(t)=0$.
  \item There exist constants $r_0>0$ and $K \ge 1$ such that
    \begin{equation*}
      K^{-1}r^t \le \HH^t|_E\bigl( B(x,r) \bigr) \le Kr^t
    \end{equation*}
    whenever $x \in E$, $0<r<r_0$, and $P(t)=0$.
  \end{enumerate}
\end{corollary}

\section{Semiconformal Moran construction} \label{sec:congruent}

In a tractable CMC, we require that the relative positions of the sets
$X_\iii$, $\iii \in I^*$, follow the rule given in
\eqref{eq:tractable}. The only restriction for the shapes of these sets
comes from \ref{M2} and \ref{M3}. Assuming more on the shape, we are
able to prove that the Hausdorff dimension and the upper Minkowski
dimension of the limit set coincide and if the uniform ball
condition is satisfied then the dimension of the intersection of
incomparable cylinder sets is small. We say that a CMC is
\emph{semiconformal} if there is a constant $C^* \ge 1$ such that
\begin{equation} \label{eq:congru2}
  \frac{\dist(X_{\hhh\iii}, X_{\hhh\jjj})}{\diam(X_\hhh)} \le
  C^*\frac{\dist(X_{\kkk\iii}, X_{\kkk\jjj})}{\diam(X_\kkk)}
\end{equation}
whenever $\hhh,\kkk,\iii,\jjj \in I^*$. This property implies that the
limit set is ``approximately self-similar''.
Observe that \eqref{eq:congru2} is equivalent to the existence of a
constant $C \ge 1$ for which
\begin{equation} \label{eq:congruent}
  C^{-1}\diam(X_\hhh)\dist(X_\iii,X_\jjj) \le
  \dist(X_{\hhh\iii},X_{\hhh\jjj}) \le
  C\diam(X_\hhh)\dist(X_\iii,X_\jjj)
\end{equation}
whenever $\hhh,\iii,\jjj \in I^*$.
We notice immediately that a semiconformal CMC is tractable which
indicates, for example, that the
finite clustering property and the uniform finite clustering
property are equivalent.

Let us first introduce natural mappings for a semiconformal CMC.

\begin{lemma} \label{thm:kuvaukset}
  If a CMC is semiconformal then for each $\iii \in I^*$ there exists a
  mapping $\fii_\iii \colon E \to E$ such that $\fii_\iii\bigl(
  \pi(\hhh) \bigr) = \pi(\iii\hhh)$ as $\hhh \in I^\infty$ and
  \begin{equation*}
    C^{-1}\diam(X_\iii)|x-y| \le |\fii_\iii(x) - \fii_\iii(y)| \le
    C\diam(X_\iii)|x-y|
  \end{equation*}
  whenever $x,y \in E$.
\end{lemma}

\begin{proof}
  Fix $\iii \in I^*$ and $\hhh,\kkk \in I^\infty$. Take $\eps>0$ and
  using Lemma \ref{thm:rajanolla}, choose $n \in \N$ such that
  $\diam(X_{\iii(\hhh|_n)}) + \diam(X_{\iii(\kkk|_n)}) < \eps$. Now,
  using \eqref{eq:congruent}, we have
  \begin{equation} \label{eq:onkuvaus}
  \begin{split}
    |\pi(\iii\hhh) - \pi(\iii\kkk)| &\le \diam(X_{\iii(\hhh|_n)}) +
    \dist(X_{\iii(\hhh|_n)},X_{\iii(\kkk|_n)}) +
    \diam(X_{\iii(\kkk|_n)}) \\
    &\le C\diam(X_\iii)\dist(X_{\hhh|_n},X_{\kkk|_n}) + \eps \\
    &\le C\diam(X_\iii)|\pi(\hhh) - \pi(\kkk)| + \eps.
  \end{split}
  \end{equation}
  On the other hand,
  choosing $n \in \N$ such that $\diam(X_{\hhh|_n}) +
  \diam(X_{\kkk|_n}) < \eps$, we get similarly
  \begin{align*}
    |\pi(\iii\hhh) - \pi(\iii\kkk)| &\ge
    \dist(X_{\iii(\hhh|_n)},X_{\iii(\kkk|_n)}) \\
    &\ge C^{-1}\diam(X_\iii)\dist(X_{\hhh|_n},X_{\kkk|_n}) \\
    &\ge C^{-1}\diam(X_\iii)\bigl( |\pi(\hhh) - \pi(\kkk)| -
    \diam(X_{\hhh|_n}) - \diam(X_{\kkk|_n}) \bigr) \\
    &\ge C^{-1}\diam(X_\iii)|\pi(\hhh) - \pi(\kkk)| -
    C^{-1}\diam(X_\iii)\eps.
  \end{align*}
  The claim follows now by letting $\eps \downarrow 0$ since
  according to \eqref{eq:onkuvaus}, we may define a mapping $\fii_\iii
  \colon E \to E$ by setting $\fii_\iii\bigl( \pi(\hhh) \bigr) =
  \pi(\iii\hhh)$ as $\hhh \in I^\infty$.
\end{proof}

It follows that the measure of Proposition \ref{thm:semileikkaus} is
semiconformal on a semiconformal CMC satisfying the finite clustering
property in the following sense.

\begin{lemma} \label{thm:nollaleikkaus}
  If a semiconformal CMC satisfies the finite
  clustering property, $P(t)=0$, and $m$ is the measure of Proposition
  \ref{thm:semileikkaus} then
  \begin{equation*}
    m\bigl( \fii_\iii(E) \cap \fii_\jjj(E) \bigr) = 0
  \end{equation*}
  whenever $\iii \bot \jjj$. Here $\fii_\iii$, $\iii \in I^*$, are the
  mappings of Lemma \ref{thm:kuvaukset}.
\end{lemma}

\begin{proof}
  Since Lemma \ref{thm:kuvaukset} clearly implies that $\diam\bigl(
  \fii_\iii(E) \bigr)$ is proportional to 
  $\diam(X_\iii)$, the CMC formed by the sets $\fii_\iii(E)$, $\iii
  \in I^*$, has the same topological pressure as the original
  CMC. Notice that $\diam(E)>0$ by the finite clustering property. By
  the uniqueness of the invariant semiconformal measure on $I^\infty$,
  also the semiconformal measures determined by these CMC's on
  $I^\infty$ are the same. Noting that the finite clustering property
  remains satisfied in the new setting and trivially $\fii_\iii(E) \cap E
  = \pi([\iii])$ 
  for each $\iii \in I^*$, Lemma \ref{thm:technical_lemma} implies the
  bounded overlapping property. By the semiconformality, it is evident that
  for each $\iii,\jjj \in I^*$ and $\hhh \in I^\infty$ it holds that
  $\pi(\iii\hhh) \in \fii_{\iii\jjj}(E)$ whenever $\pi(\hhh) \in
  \fii_\jjj(E)$ and hence
  Proposition \ref{thm:semileikkaus} completes the proof.
\end{proof}

Using the mappings of Lemma \ref{thm:kuvaukset}, we are able to prove
that the Hausdorff dimension
and the upper Minkowski dimension of the limit set of a semiconformal CMC
coincide even without assuming the ball condition.

\begin{theorem} \label{thm:sama_dim}
  If a CMC is semiconformal and $t=\dimh(E)$ then $\dimm(E)=t$ and
  $\HH^t(E)<\infty$.
\end{theorem}

\begin{proof}
  We may assume that $\diam(E)>0$.
  Let $\fii_\iii$, $\iii \in I^*$, be the mappings of Lemma
  \ref{thm:kuvaukset}. Notice that, using \ref{M2}, there exists a
  constant $\delta>0$ such that
  \begin{equation}
    \label{eq:delta-arvio}
    \diam(X_{\iii i}) \ge \delta \diam(X_\iii)
  \end{equation}
  whenever $\iii \in I^*$ and $i \in I$. Take $x_0 \in E$, $\hhh \in
  I^\infty$ such that $x_0=\pi(\hhh)$, and $0<r<C\diam(E)^2$.
  Then choose $n \in \N$ such that $\hhh|_n \in Z\bigl( C^{-1}
  \diam(E)^{-1} r \bigr)$. Since $x_0 = \fii_{\hhh|_n}\bigl(
  \pi(\sigma^n(\hhh)) \bigr)$, we have
  \begin{align*}
    |x_0 - \fii_{\hhh|_n}(y)| &\le C\diam(X_{\hhh|_n}) \bigl|
    \pi\bigl( \sigma^n(\hhh) \bigr) - y \bigr| \\
    &\le C\diam(X_{\hhh|_n})\diam(E) < r
  \end{align*}
  for every $y \in E$. Hence
  \begin{equation*}
    \fii_{\hhh|_n}(E) \subset E \cap B(x_0,r).
  \end{equation*}
  On the other hand, using \eqref{eq:delta-arvio},
  \begin{align*}
    |\fii_{\hhh|_n}(x) - \fii_{\hhh|_n}(y)| &\ge
    C^{-1}\diam(X_{\hhh|_n})|x-y| \\
    &\ge C^{-2}\diam(E)^{-1}\delta r|x-y|
  \end{align*}
  whenever $x,y \in E$. Therefore for each $x_0 \in E$ and
  $0<r<C\diam(E)^2$ there is a mapping $g \colon E \to E \cap
  B(x_0,r)$ and a constant $a=C^{-2}\diam(E)^{-1}\delta$ such that
  \begin{equation*}
    |g(x) - g(y)| \ge ar|x-y|
  \end{equation*}
  whenever $x,y \in E$. The claim follows now from \cite[Theorem
  4]{Falconer1989}.
\end{proof}

The following simple proposition shows the bi-Lipschitz invariance of
a semiconformal CMC. Therefore the collection of all semiconformal
CMC's is sufficiently large. Observe that despite of this property
the geometry of the limit set may change a lot under a bi-Lipschitz
map, see \cite[Lemma 3.2]{MartinMattila2000}.

\begin{proposition}
  If $\{ X_\iii : \iii \in I^* \}$ is a semiconformal CMC with $E$ as a
  limit set and $h \colon \R^d \to \R^d$ is a bi-Lipschitz mapping
  then $\{ h(X_\iii) : \iii \in I^* \}$ is a semiconformal CMC with $h(E)$
  as a limit set.
\end{proposition}

\begin{proof}
  Fix constants $a,b > 0$ such that
  \begin{equation*}
    a|x-y| \le |h(x) - h(y)| \le b|x-y|
  \end{equation*}
  for every $x,y \in X$. The condition \ref{M1} is clearly satisfied
  and since $a\diam(X_\iii) \le \diam\bigl( h(X_\iii) \bigr) \le
  b\diam(X_\iii)$ as $\iii \in I^*$ and $a\dist(X_\iii,X_\jjj) \le
  \dist\bigl( h(X_\iii),h(X_\jjj) \bigr) \le b\dist(X_\iii,X_\jjj)$ as
  $\iii,\jjj \in I^*$, also the conditions \ref{M2}, \ref{M3}, and
  \eqref{eq:congruent} are satisfied. The proof is finished.
\end{proof}

Examining the method used in \cite[Theorem 2.1]{Schief1994}, one is
easily convinced by the usefulness of the set of symbols $W$
defined by
\begin{equation} \label{eq:tuplawee}
\begin{split}
  W(\iii) = \bigl\{ \jjj \in I^* : \;\,&\jjj' \in
  Z\bigl( \diam(X_{\iii'}) \bigr) \text{ and
  } \\ &\dist(X_{\iii'},X_{\jjj'}) \le 3\diam(X_{\iii'}),
  \text{ where } \\ &\iii' = \sigma^{|\iii\land\jjj|}(\iii) \text{ and }
  \jjj' = \sigma^{|\iii\land\jjj|}(\jjj) \bigr\}
\end{split}
\end{equation}
as $\iii \in I^*$. See also \cite[\S 2]{LauRaoYe2001} and \cite[\S
3]{PeresRamsSimonSolomyak2001}.
Notice that $\iii \in W(\iii)$.
The constant $3$ in \eqref{eq:tuplawee} is somewhat arbitrary. The
reader can easily see that any constant strictly larger than $2$ would
suffice. Let us next prove two technical lemmas.

\begin{lemma} \label{thm:tuplawee}
  Given a CMC, the set $W(\iii)$ is incomparable for every
  $\iii \in I^*$. Furthermore, if $\jjj \in W(\iii)$ then
  \begin{equation*}
    D^{-3}\min_{i \in I}\diam(X_i)\diam(X_\iii) \le \diam(X_\jjj) \le
    D^2\diam(X_\iii).
  \end{equation*}
\end{lemma}

\begin{proof}
  Fix $\iii \in I^*$. Observe that if $\iii \ne \jjj \in W(\iii)$ then
  clearly $\iii \bot \jjj$. Take $\jjj,\hhh \in W(\iii)$. If now
  $|\jjj\land\iii| < |\hhh\land\iii|$, it must be $\jjj\bot\hhh$ since
  otherwise $\jjj=\iii\land\jjj$, which contradicts with the first
  observation. If $|\jjj\land\iii| = |\hhh\land\iii| =: k$ then
  $\sigma^k(\jjj), \sigma^k(\hhh) \in
  Z\bigl( \diam(X_{\sigma^k(\iii)}) \bigr)$ and hence
  $\jjj\bot\hhh$.

  To prove the second claim, fix $\iii \in I^*$, take $\jjj \in
  W(\iii)$, and denote $\iii'=\sigma^{|\iii\land\jjj|}(\iii)$ and
  $\jjj'=\sigma^{|\iii\land\jjj|}(\jjj)$. Since $\jjj' \in Z\bigl(
  \diam(X_{\iii'}) \bigr)$, we have, using \ref{M2},
  \begin{equation*}
    \diam(X_{\iii'}) \ge \diam(X_{\jjj'}) \ge D^{-1}\min_{i \in
    I}\diam(X_i)\diam(X_{\iii'}).
  \end{equation*}
  Therefore, according to \ref{M2},
  \begin{align*}
    \diam(X_\jjj) &\ge D^{-1} \diam(X_{\iii\land\jjj})
    \diam(X_{\jjj'}) \\
    &\ge D^{-2} \min_{i \in I} \diam(X_i) \diam(X_{\iii\land\jjj})
    \diam(X_{\iii'}) \\
    &\ge D^{-3} \min_{i \in I} \diam(X_i) \diam(X_\iii)
  \end{align*}
  and
  \begin{align*}
    \diam(X_\jjj) &\le D\diam(X_{\iii\land\jjj})\diam(X_{\jjj'}) \\
    &\le D\diam(X_{\iii\land\jjj})\diam(X_{\iii'}) \le D^2\diam(X_\iii).
  \end{align*}
  The proof is finished.
\end{proof}

\begin{lemma} \label{thm:tractable_finite}
  If a semiconformal CMC satisfies the finite clustering property then
  \begin{equation*}
    \sup_{\iii \in I^*} \# W(\iii) < \infty.
  \end{equation*}
\end{lemma}

\begin{proof}
  Suppose $\fii_\iii$, $\iii \in I^*$, are the mappings of Lemma
  \ref{thm:kuvaukset}, $P(t)=0$, and $m = \mu \circ \pi^{-1}$ is the
  measure of Proposition \ref{thm:semileikkaus}. According to
  Corollary \ref{thm:tractable_corollary} and Theorems
  \ref{thm:bounded_separation} and
  \ref{thm:regular}, there exists a constant $K \ge 1$ such that for
  every $x \in E$ and $r>0$
  \begin{equation*}
    m\bigl( B(x,r) \bigr) \le Kr^t.
  \end{equation*}
  Fix $\iii \in I^*$, take $\jjj \in W(\iii)$, and denote $\iii' =
  \sigma^{|\iii\land\jjj|}(\iii)$ and $\jjj' =
  \sigma^{|\iii\land\jjj|}(\jjj)$. Since $\jjj \in W(\iii)$ and $\jjj'
  \in Z\bigl( \diam(X_{\iii'}) \bigr)$, we have
  $\dist(X_{\iii'},X_{\jjj'}) \le \diam(X_{\iii'})$ and
  \begin{align*}
    \dist(X_\iii,X_\jjj) &\le C \diam(X_{\iii\land\jjj})
    \dist(X_{\iii'},X_{\jjj'}) \\
    &\le 3C \diam(X_{\iii\land\jjj})\diam(X_{\iii'}) \le 3CD \diam(X_\iii).    
  \end{align*}
  Using Lemma \ref{thm:tuplawee}, we obtain
  \begin{align*}
    X_\jjj &\subset B\bigl( x,\diam(X_\iii) + 3CD\diam(X_\iii) +
    \diam(X_\jjj) \bigr) \\
    &\subset B\bigl( x,(1+3CD+D^2)\diam(X_\iii) \bigr)
  \end{align*}
  for a point $x \in \pi([\iii])$ whenever $\jjj \in W(\iii)$. Hence
  \begin{align*}
    m\biggl( \bigcup_{\jjj \in W(\iii)} X_\jjj \biggr) &\le m\bigl(
    B\bigl( x,(1+3CD+D^2)\diam(X_\iii) \bigr) \bigr) \\
    &\le K(1+3CD+D^2)^t\diam(X_\iii)^t.
  \end{align*}
  Since, on the other hand, we have a constant $c \ge 1$ such that
  \begin{align*}
    m\biggl( \bigcup_{\jjj \in W(\iii)} X_\jjj \biggr)
    &\ge m\biggl( \bigcup_{\jjj \in W(\iii)} \fii_\jjj(E) \biggr) 
    = \sum_{\jjj \in W(\iii)} m\bigl( \fii_\jjj(E) \bigr) \\
    &\ge \sum_{\jjj \in W(\iii)} \mu([\jjj])
    \ge c^{-1}\sum_{\jjj \in W(\iii)} \diam(X_\jjj)^t \\
    &\ge \# W(\iii)c^{-1}D^{-3t}\min_{i \in I}\diam(X_i)^t\diam(X_\iii)^t
  \end{align*}
  using Lemmas \ref{thm:nollaleikkaus} and \ref{thm:tuplawee}, we
  conclude
  \begin{equation*}
    \# W(\iii) \le \frac{cKD^{3t}(1+3CD+D^2)^t}{\min_{i \in I}\diam(X_i)^t}
  \end{equation*}
  whenever $\iii \in I^*$.
\end{proof}

The following theorem generalizes a crucial point of \cite[Theorem
2.1]{Schief1994} into the setting of CMC's. See also \cite[Theorem
3.3]{LauRaoYe2001} and \cite[\S 3]{PeresRamsSimonSolomyak2001}.

\begin{theorem} \label{thm:schief}
  If a semiconformal CMC satisfies the finite clustering property then
  there are a constant $\delta>0$ and a symbol $\hhh \in I^*$ such
  that
  \begin{equation*}
    \dist(X_{\iii\hhh},X_{\jjj\hhh}) > \delta\bigl( \diam(X_\iii) +
    \diam(X_\jjj) \bigr)
  \end{equation*}
  whenever $\iii\bot\jjj$.
\end{theorem}

\begin{proof}
  Using Lemma \ref{thm:tractable_finite}, we choose $\hhh \in I^*$
  such that $\# W(\hhh) = \sup_{\iii \in I^*} \# W(\iii)$. Therefore
  clearly
  \begin{equation*}
    \#\{ \iii\jjj : \jjj \in W(\hhh) \} = \# W(\hhh) \ge \# W(\iii\hhh)
  \end{equation*}
  for every $\iii \in I^*$. Since it follows immediately from the
  definition \eqref{eq:tuplawee} that
  \begin{equation*}
    \{ \iii\jjj : \jjj \in W(\hhh) \} \subset W(\iii\hhh),
  \end{equation*}
  we infer
  \begin{equation} \label{eq:maxwee}
    W(\iii\hhh) = \{ \iii\jjj : \jjj \in W(\hhh) \}
  \end{equation}
  whenever $\iii \in I^*$.

  Take next $\iii,\jjj \in I^*$ such that $\iii\bot\jjj$ and denote
  $\iii' = \sigma^{|\iii\land\jjj|}(\iii)$ and $\jjj' =
  \sigma^{|\iii\land\jjj|}(\jjj)$. 
  Let $y_{\jjj'} = \pi(\kkk) \in X_{\jjj'\hhh}$, where $\kkk \in
  [\jjj'\hhh]$, and choose $k \in \N$ such that $\kkk|_k \in Z\bigl(
  \diam(X_{\iii'\hhh}) \bigr)$. Since $\kkk|_1 = \jjj'|_1 \ne
  \iii'|_1$, we have, using \eqref{eq:maxwee},
  \begin{equation*}
    \kkk|_k \notin W(\iii'\hhh).
  \end{equation*}
  Hence the definition \eqref{eq:tuplawee} yields
  $\dist(X_{\kkk|_k},X_{\iii'\hhh}) > 3\diam(X_{\iii'\hhh})$. Since
  $y_{\jjj'} \in X_{\kkk|_k}$, we also have
  $\dist(y_{\jjj'},X_{\iii'\hhh}) > 3\diam(X_{\iii'\hhh})$. Similarly,
  changing the roles of $\iii$ and $\jjj$ above, we find $y_{\iii'}
  \in X_{\iii'\hhh}$ such that $\dist(y_{\iii'},X_{\jjj'\hhh}) >
  3\diam(X_{\jjj'\hhh})$. This implies that
  \begin{align*}
    |y_{\iii'} - y_{\jjj'}| &\ge 3\max\{
    \diam(X_{\iii'\hhh}),\diam(X_{\jjj'\hhh}) \} \\
    &\ge \tfrac{3}{2} \bigl( \diam(X_{\iii'\hhh}) + \diam(X_{\jjj'\hhh})
    \bigr).
  \end{align*}
  Since, on the other hand,
  \begin{equation*}
    |y_{\iii'} - y_{\jjj'}| \le \diam(X_{\iii'\hhh}) +
     \dist(X_{\iii'\hhh},X_{\jjj'\hhh}) + \diam(X_{\jjj'\hhh}),
  \end{equation*}
  we infer
  \begin{equation*}
    \dist(X_{\iii'\hhh},X_{\jjj'\hhh}) \ge \tfrac12 \bigl(
    \diam(X_{\iii'\hhh}) + \diam(X_{\jjj'\hhh}) \bigr).
  \end{equation*}
  Thus, using \eqref{eq:congruent} and \ref{M2},
  \begin{align*}
    \dist\bigl( X_{\iii\hhh},X_{\jjj\hhh} \bigr) &\ge C^{-1}
    \diam(X_{\iii\land\jjj}) \dist\bigl(
    X_{\iii'\hhh},X_{\jjj'\hhh} \bigr) \\
    &\ge (2C)^{-1} \diam(X_{\iii\land\jjj}) \bigl(
    \diam(X_{\iii'\hhh}) + \diam(X_{\jjj'\hhh}) \bigr) \\
    &\ge (2CD)^{-1} \bigl( \diam(X_{\iii\hhh}) + \diam(X_{\jjj\hhh})
    \bigr) \\ 
    &\ge (2CD^2)^{-1} \diam(X_\hhh) \bigl( \diam(X_\iii) +
    \diam(X_\jjj) \bigr)
  \end{align*}
  whenever $\iii\bot\jjj$. Therefore, choosing $\delta = (3CD^2)^{-1}
  \diam(X_\hhh)$, we have finished the proof.
\end{proof}

As a corollary, we notice that for a semiconformal Moran construction, we
may choose the balls in the ball condition to be centered at
$E$ and placed in such manner that also larger collections (than
required in the definition) of them are disjoint.

\begin{corollary} \label{thm:schief_cor}
  If a semiconformal CMC satisfies the ball condition then
  there are a constant $\delta>0$ and a point $x \in E$ such
  that
  \begin{equation*}
    B\bigl( \fii_\iii(x),\delta \diam(X_\iii) \bigr) \cap B\bigl(
    \fii_\jjj(x),\delta \diam(X_\jjj) \bigr) = \emptyset
  \end{equation*}
  whenever $\iii\bot\jjj$. Here $\fii_\iii$, $\iii \in I^*$, are the
  mappings of Lemma \ref{thm:kuvaukset}.
\end{corollary}

\begin{proof}
  Assuming that $\delta>0$ and $\hhh \in I^*$ are as in Theorem
  \ref{thm:schief}, the claim follows immediately from Theorems
  \ref{thm:bounded_separation} and \ref{thm:schief} by choosing $x \in
  \pi([\hhh])$.
\end{proof}

Compare the following improvement of Lemma \ref{thm:nollaleikkaus} to
\cite[Theorem 3.3]{Moran1999} and \cite[Theorem 1.6]{LauXu2000}.

\begin{proposition} \label{thm:leikkaus_pieni}
  If a semiconformal CMC satisfies the ball condition then
  \begin{equation*}
    \dimh\bigl( \fii_\iii(E) \cap \fii_\jjj(E) \bigr) < \dimh(E)
  \end{equation*}
  whenever $\iii\bot\jjj$. Here $\fii_\iii$, $\iii \in I^*$, are the
  mappings of Lemma \ref{thm:kuvaukset}.
\end{proposition}

\begin{proof}
  Let $\delta > 0$ and $\hhh \in I^*$ be as in Theorem
  \ref{thm:schief} and define
  \begin{equation*}
    A = \bigcup_{\kkk \in I^*} \fii_\kkk\bigl( \pi([\hhh]) \bigr).
  \end{equation*}
  According to Theorem \ref{thm:schief}, we have $\fii_\iii\bigl(
  \pi([\hhh]) \bigr) \cap \fii_\jjj\bigl( \pi([\hhh]) \bigr) =
  \emptyset$ whenever $\iii\bot\jjj$, and hence also
  \begin{equation*}
    \fii_\iii(A) \cap \fii_\jjj(A) = \emptyset
  \end{equation*}
  as $\iii\bot\jjj$. Thus we get
  \begin{align*}
    \fii_\iii(E) \cap \fii_\jjj(E) &= \bigl( \fii_\iii(E \setminus A)
    \cap \fii_\jjj(A) \bigr) \cup \bigl( \fii_\iii(E) \cap \fii_\jjj(E
    \setminus A) \bigr) \\
    &\subset \fii_\iii(E \setminus A) \cup \fii_\jjj(E \setminus A)
  \end{align*}
  whenever $\iii \bot \jjj$ from which the Lipschitz continuity
  implies
  \begin{align*}
    \dimh\bigl( \fii_\iii(E) \cap \fii_\jjj(E) \bigr) &\le \dimh\bigl(
    \fii_\iii(E \setminus A) \cup \fii_\jjj(E \setminus A) \bigr) \\
    &\le \dimh(E \setminus A).
  \end{align*}
  Obviously, $\{ X_\iii : \iii \in (I^{|\hhh|})^* \}$ is a CMC having
  $E$ as a limit set, whereas $E \setminus A$ is contained in the limit set
  $F$ of the subconstruction $\{ X_\iii : \iii \in (I^{|\hhh|} \setminus
  \{ \hhh \})^* \}$. Since it is evident that both of these CMC's
  satisfy the uniform finite clustering property, Lemma
  \ref{thm:Pj_pieni} and Theorem \ref{thm:regular} imply that $\dimh(F )
  < \dimh(E)$. Consequently, $\dimh(E \setminus A) < \dimh(E)$ and the
  proof is finished.
%
\end{proof}

We shall finish this section with the following observation.

\begin{proposition} \label{thm:congruent_controlled}
  Suppose a collection of compact sets with positive diameter $\{
  X_\iii \subset \R^d : \iii \in I^* \}$ satisfies the following four
  conditions:
  \begin{labeledlist}{C}
  \item $X_{\iii i} \subset X_\iii$ as $\iii \in I^*$ and $i \in I$,
    \label{A1}
  \item there exist $\iii,\jjj \in I^*$ such that $X_\iii \cap X_\jjj
    = \emptyset$, \label{A2}
  \item $\lim_{n \to \infty} \diam(X_{\iii|_n}) = 0$ for every $\iii
    \in I^\infty$, \label{A3}
  \item there exists a constant $C \ge 1$ such that
    \begin{equation*}
      C^{-1}\diam(X_\hhh)\dist(X_\iii,X_\jjj) \le
      \dist(X_{\hhh\iii},X_{\hhh\jjj}) \le
      C\diam(X_\hhh)\dist(X_\iii,X_\jjj)
    \end{equation*}
    whenever $\hhh,\iii,\jjj \in I^*$. \label{A4}
  \end{labeledlist}
  Then the collection is a semiconformal CMC.
\end{proposition}

\begin{proof}
  It suffices to prove \ref{M2} and \ref{M3}. To show \ref{M2},
  observe first that the assumptions \ref{A1} and \ref{A3}
  guarantee the existence of the limit set $E$ and the claim in Lemma
  \ref{thm:kuvaukset} follows from the assumptions \ref{A1},
  \ref{A3}, and \ref{A4}. Notice also that the assumption \ref{A2}
  implies immediately that $\diam(E)>0$. Let $\fii_\iii$, $\iii \in I^*$,
  be the mappings of Lemma \ref{thm:kuvaukset}. Then
  \begin{equation*}
    \diam\bigl( \fii_\iii(E) \bigr) \ge |\fii_\iii(x) - \fii_\iii(y)|
    \ge C^{-1}\diam(X_\iii)|x-y|
  \end{equation*}
  whenever $x,y \in E$ and it follows that
  \begin{equation}
    \label{eq:ekakohta}
    \diam(X_\iii) \le C\diam(E)^{-1}\diam\bigl( \fii_\iii(E) \bigr)
  \end{equation}
  for every $\iii \in I^*$. Since
  \begin{align*}
    \diam\bigl( \fii_{\iii\jjj}(E) \bigr) &= \sup_{x,y \in E}\bigl|
    \fii_\iii\bigl( \fii_\jjj(x) \bigr) - \fii_\iii\bigl( \fii_\jjj(y)
    \bigr) \bigr| \\
    &\le C^2\diam(X_\iii)\diam(X_\jjj) \sup_{x,y \in E}|x-y|
  \end{align*}
  whenever $\iii,\jjj \in I^*$, we get, by \eqref{eq:ekakohta}, that
  \begin{align*}
    \diam(X_{\iii\jjj}) &\le C\diam(E)^{-1}\diam\bigl( \fii_{\iii\jjj}(E)
    \bigr) \\
    &\le C^3 \diam(X_\iii)\diam(X_\jjj)
  \end{align*}
  whenever $\iii,\jjj \in I^*$.
  On the other hand,
  \begin{equation*}
    \diam\bigl( \fii_\iii(E) \bigr) = \sup_{x,y \in E} |\fii_\iii(x) -
    \fii_\iii(y)| \le C\diam(X_\iii)\sup_{x,y\in E}|x-y|
  \end{equation*}
  implies that
  \begin{equation}
    \label{eq:tokakohta}
    \diam(X_\iii) \ge C^{-1}\diam(E)^{-1}\diam\bigl( \fii_\iii(E) \bigr)
  \end{equation}
  for every $\iii \in I^*$. Since
  \begin{align*}
    \diam\bigl( \fii_{\iii\jjj}(E) \bigr) &\ge \bigl|
    \fii_\iii\bigl( \fii_\jjj(x) \bigr) - \fii_\iii\bigl( \fii_\jjj(y)
    \bigr) \bigr| \\
    &\ge C^{-2}\diam(X_\iii)\diam(X_\jjj)|x-y|
  \end{align*}
  whenever $x,y \in E$ and $\iii,\jjj \in I^*$, we get, by
  \eqref{eq:tokakohta}, that
  \begin{align*}
    \diam(X_{\iii\jjj}) &\ge C^{-1}\diam(E)^{-1}\diam\bigl(
    \fii_{\iii\jjj}(E) \bigr) \\
    &\ge C^{-3}\diam(X_\iii)\diam(X_\jjj)
  \end{align*}
  whenever $\iii,\jjj \in I^*$.

  Let us then show \ref{M3}. Denote $M_n = \max_{\iii \in I^n}
  \diam(X_\iii)$ as $n \in \N$ and choose $\iii_1,\iii_2,\ldots \in
  I^\infty$ such that
  \begin{equation*}
    M_n = \diam(X_{\iii_n|_n})
  \end{equation*}
  for every $n \in \N$. By the compactness of $I^\infty$, the sequence
  $\{ \iii_n \}_{n \in \N}$ has a converging subsequence. Let $\iii
  \in I^\infty$ be the limit point of such a subsequence. Now for each
  $j \in \N$ there is $n(j) \in \N$ such that $n(j) \ge j$ and
  $\iii_{n(j)} \in [\iii|_j]$. Since $\iii_{n(j)}|_j = \iii|_j$ for
  all $j \in \N$, we have, using \ref{A1} and \ref{A3},
  \begin{align*}
    M_{n(j)} &= \diam(X_{\iii_{n(j)}|_{n(j)}}) \\
    &\le \diam(X_{\iii_{n(j)}|_j}) = \diam(X_{\iii|_j}) \to 0
  \end{align*}
  as $j \to \infty$. The proof is finished by choosing $j \in \N$ such
  that $M_{n(j)} < C^{-3}$.
\end{proof}

\section{Semiconformal iterated function system} \label{sec:congruent_IFS}

We assume that for each $i \in I$ there is a contractive injection
$\fii_i \colon \Omega \to \Omega$ defined on an open subset $\Omega$
of $\R^d$ and that there also exists a closed and nonempty $X \subset
\Omega$ satisfying
\begin{equation}
  \label{eq:X_forward_invariant}
  \bigcup_{i \in I} \fii_i(X) \subset X.
\end{equation}
Here the \emph{contractivity} of $\fii_i$ means that there is a
constant $0<s_i<1$ such that
\begin{equation}
  \label{eq:contraction}
  |\fii_i(x) - \fii_i(y)| \le s_i|x-y|
\end{equation}
whenever $x,y \in \Omega$. The collection $\{ \fii_i : i \in I \}$ is
then called an \emph{iterated function system (IFS)}. As shown in
\cite[\S 3]{Hutchinson1981}, an elegant application of the Banach
fixed point theorem implies the existence of a unique compact and
nonempty set $E \subset X$ for which
\begin{equation*}
  E = \bigcup_{i \in I} \fii_i(E).
\end{equation*}
Such a set $E$ is called an \emph{invariant set} (for the corresponding
IFS). As a side note, it is not necessary to require the mappings
$\fii_i$ to be injective to ensure the existence of the
invariant set. However, under this additional assumption, it follows
from Brouwer's domain invariance theorem \cite[Theorem
IV.7.4]{Dold1972} that $\fii_i(U)$ is open whenever $U$ is open.

Observe that the set $X$ can be chosen to be a closed neighborhood
of the invariant set $E$. Indeed, we fix $0 < \eps <
\dist(E,\R^d\setminus\Omega)$ (if $\Omega = \R^d$, any positive $\eps$
will do) and take
\begin{equation*}
  X = \{ x \in \Omega : |x-a| \le \eps \text{ for some } a \in E \}.
\end{equation*}
The validity of \eqref{eq:X_forward_invariant} is then a consequence of
the easily proven fact that
\begin{equation}
  \label{eq:kutistava_dist}
  \dist\bigl( \fii_i(A),E \bigr) \le s_i \dist(A,E)
\end{equation}
whenever $A \subset \Omega$ and $i \in I$.

We say that an IFS is \emph{tractable} if there exists a compact set
$A \subset \Omega$ and a constant $C > 0$ such that for each $\iii \in
I^*$
\begin{equation*}
  |\fii_\iii(x) - \fii_\iii(y)| \le C\diam\bigl( \fii_\iii(A) \bigr)|x-y|
\end{equation*}
whenever $x,y \in A$ and it defines a CMC in this situation, that is,
the collection 
$\{ \fii_\iii(A) : \iii \in I^* \}$ is a CMC. The limit set of such a
CMC is clearly $E$. Here $\fii_\iii = \fii_{i_1}
\circ \cdots \circ \fii_{i_n}$ as $\iii = (i_1,\ldots,i_n) \in I^n$
and $n \in \N$.

\begin{lemma} \label{thm:IFS_controlled_tractable}
  A tractable IFS defines a tractable CMC.
\end{lemma}

\begin{proof}
  Choose a compact set $A \subset X$ such that the collection $\{
  \fii_\iii(A) : \iii \in I^* \}$ is a CMC. Then
  \begin{equation*}
    \dist\bigl( \fii_{\iii\hhh}(A),\fii_{\iii\kkk}(A) \bigr) \le
    C\diam\bigl( \fii_\iii(A) \bigr) \dist\bigl(
    \fii_\hhh(A),\fii_\kkk(A) \bigr),
  \end{equation*}
  which implies \eqref{eq:tractable} and finishes the proof.
\end{proof}

Furthermore, we say that an IFS is \emph{semiconformal} if the invariant
set $E$ has positive diameter and there are
constants $0 < \as_\iii \le \ys_\iii < 1$, $\iii \in I^*$, and $D \ge
1$ for which $\ys_\iii \le D\as_\iii$ as $\iii \in I^*$ and
\begin{equation}
  \label{eq:congruent_function}
  \as_\iii|x-y| \le |\fii_\iii(x) - \fii_\iii(y)| \le \ys_\iii|x-y|
\end{equation}
whenever $x,y \in \Omega$. For an interesting class of quasiregular
mappings which admit uniform control of the distortion with respect to
iteration, the reader is referred to \cite{MartinMayerPeltonen2006}.

The following lemma shows that a semiconformal
IFS defines a semiconformal CMC. The natural question whether the converse
holds raises from Lemma \ref{thm:kuvaukset}. Sufficient geometric
conditions on the limit set for the positive answer are provided in
\cite{AlestaloTrotsenkoVaisala2003}. See also \cite[Example
6.2]{VaisalaVuorinenWallin1994}.

\begin{lemma} \label{thm:congruent_ifs}
  If $\{ \fii_i : i \in I \}$ is a semiconformal IFS and a compact set $A$
  with positive diameter
  satisfies $\fii_i(A) \subset A$ for every $i \in I$ then $\{
  \fii_\iii(A) : \iii \in I^* \}$ is a semiconformal CMC.
  In particular, a semiconformal IFS is tractable.
  Furthermore, the mappings $\fii_\iii|_E$, $\iii \in I^*$, are the
  mappings of Lemma \ref{thm:kuvaukset}.
\end{lemma}

\begin{proof}
  To be able to use Proposition \ref{thm:congruent_controlled}, we
  have to verify the required assumptions \ref{A1}--\ref{A4}.
  Observe first that \ref{A1} is clearly satisfied and the positivity
  of $\diam(E)$ implies \ref{A2}. Notice also that the sets
  $\fii_\iii(A)$, $\iii \in I^*$, are compact with positive
  diameter. Since for fixed $\iii \in I^*$, we have
  $\as_\iii\diam(A) \le \diam\bigl( \fii_\iii(A) \bigr) \le
  \ys_\iii\diam(A)$ by \eqref{eq:congruent_function}, it follows that
  \begin{equation}
    \label{eq:diam_bi-lip}
    C^{-1}\diam\bigl( \fii_\iii(A) \bigr)|x-y| \le |\fii_\iii(x) -
    \fii_\iii(y)| \le C\diam\bigl( \fii_\iii(A) \bigr)|x-y|,
  \end{equation}
  where $C=D\max\{ \diam(A),\diam(A)^{-1} \}$ and $x,y \in A$. Hence,
  applying \eqref{eq:contraction}
  to \eqref{eq:diam_bi-lip} several times, we get $\diam\bigl(
  \fii_\iii(A) \bigr) \le Cs_{i_1}\cdots s_{i_{|\iii|}}$, which
  implies the assumption \ref{A3}.
  Since \eqref{eq:diam_bi-lip} gives also the assumption \ref{A4},
  that is
  \begin{gather*}
    C^{-1}\diam\bigl( \fii_\iii(A) \bigr) \dist\bigl(
    \fii_\hhh(A),\fii_\kkk(A) \bigr) \le \dist\bigl(
    \fii_{\iii\hhh}(A),\fii_{\iii\kkk}(A) \bigr) \\ \le C\diam\bigl(
    \fii_\iii(A) \bigr) \dist\bigl( \fii_\hhh(A),\fii_\kkk(A) \bigr)
  \end{gather*}
  as $\hhh,\kkk \in I^*$, we have finished the proof of the first
  claim.

  The second claim follows from \eqref{eq:diam_bi-lip} by recalling
  that the invariant set $E$ has positive diameter and it satisfies
  $\fii_i(E) \subset E$ for every $i \in I$.
  The third claim follows immediately since the collection $\{
  \fii_\iii(E) : \iii \in I^* \}$ is a semiconformal CMC.
\end{proof}

We say that an IFS satisfies an \emph{open set condition
(OSC)}, if there exists a nonempty open set $U \subset \Omega$ such
that
\begin{equation*}
  \fii_\iii(U) \cap \fii_\jjj(U) = \emptyset
\end{equation*}
whenever $\iii\bot\jjj$. See \cite[Theorem III]{Moran1946}
for the motivation of the definition.
Adapting terminology from \cite{BandtHungRao2006}, we call any
such nonempty open set $U$ a \emph{feasible set} for the OSC. If there
is a feasible set intersecting $E$, we say that a \emph{strong OSC} is
satisfied. As an immediate consequence of the definition, we notice
that each nonempty open subset and each image $\fii_\iii(U)$ of a
feasible set $U$ is feasible as well. Thus, using the observation
\eqref{eq:kutistava_dist} repeatedly, we see that the OSC is
equivalent to the existence of a feasible set $U \subset X$. Recall
that $X$ is the fixed compact $\eps$-neighborhood of the invariant
set. The next lemma shows that this definition of the OSC is
equivalent to the more commonly used one, see \cite[\S
5.2]{Hutchinson1981}.

\begin{lemma} \label{thm:classical_OSC}
  An IFS satisfies the OSC exactly when
  there exists a nonempty open set $V \subset X$ such that
  \begin{equation*}
    \fii_i(V) \subset V
  \end{equation*}
  as $i \in I$ and
  \begin{equation*}
    \fii_i(V) \cap \fii_j(V) = \emptyset
  \end{equation*}
  as $i \ne j$. Furthermore, there exists a feasible set
  intersecting $E$ if and only if there exists a set $V$ as above
  such that $V \cap E \ne \emptyset$.
\end{lemma}

\begin{proof}
  Defining $V = \bigcup_{\hhh \in I^*} \fii_\hhh(U)$, where $U \subset
  X$ is a feasible set for the OSC, we clearly have $\fii_i(V) \subset
  V \subset X$ as $i \in I$. If $i \ne j$, it holds that
  \begin{equation*}
    \fii_{i\hhh}(U) \cap \fii_{j\hhh}(U) = \emptyset
  \end{equation*}
  for every $\hhh \in I^*$ and hence
  \begin{equation*}
    \biggl( \bigcup_{\hhh \in I^*} \fii_{i\hhh}(U) \biggr) \cap
    \biggl( \bigcup_{\hhh \in I^*} \fii_{j\hhh}(U) \biggr) = \emptyset.
  \end{equation*}
  Noting that the other direction is trivial we have finished the
  proof.
\end{proof}

Given IFS, we say that $A \subset \Omega$ is \emph{forwards invariant} if
$\fii_i(A) \subset A$ as $i \in I$ and \emph{backwards invariant} if
$\fii_i^{-1}(A) \subset A$ as $i \in I$. For $A \subset \Omega$ we define
\begin{equation*}
  F_A = \bigcup_{\iii \bot \jjj} \fii_\iii^{-1} \bigl( \fii_\jjj(A) \bigr)
\end{equation*}
and for a semiconformal IFS we set
\begin{equation*}
  O_A = \bigl\{ x \in \Omega : D\dist(x,A) < \dist\bigl(x,F_A \cup (\R^d
  \setminus \Omega) \bigr) \bigr\}.
\end{equation*}
Here the constant $D\ge 1$ is the same as in the definition of the
semiconformal IFS. Observe that $F_A \subset \Omega$ is backwards invariant.

\begin{proposition} \label{thm:OA_feasible}
  Suppose a given IFS is semiconformal. If $U \subset \Omega$ is a
  feasible set for the OSC then $O_U \ne \emptyset$. Furthermore, if
  there exists a set $A \subset \Omega$ such that $O_A
  \ne \emptyset$ then $O_A$ is feasible.
\end{proposition}

\begin{proof}
  Let $U \subset \Omega$ be a nonempty open set for which $\fii_\iii(U)
  \cap \fii_\jjj(U) = \emptyset$ whenever $\iii \bot \jjj$. It
  follows that $U \cap F_U = \emptyset$ and since $U$ is open, we get
  $U \subset O_U$.

  Conversely, it suffices to show that $\fii_\iii(O_A) \cap
  \fii_\jjj(O_A) = \emptyset$ as $\iii
  \bot \jjj$. Suppose contrarily that there are $\iii,\jjj \in I^*$
  and $x,y \in O_A$ such that $\iii \bot \jjj$ and $\fii_\iii(x) =
  \fii_\jjj(y) =: z$. 
  Observe that the inverse mapping $\fii_\iii^{-1} \colon
  \fii_\iii(\Omega) \to \Omega$ has a Lipschitz constant
  $\as_\iii^{-1}$ for each $\iii \in I^*$. According to
  Kirszbraun's theorem \cite[\S 2.10.43]{Federer1969}, there exists a
  Lipschitz extension $\overline{\fii}_\iii \colon \Omega \to \R^d$
  for the mapping $\fii_\iii^{-1}$ having the same Lipschitz constant.
  Since $\overline{\fii}_\iii \bigl( \fii_\jjj(A) \bigr) \subset
  F_A \cup (\R^d \setminus \Omega)$ and $x \in O_A$, we have
  \begin{align*}
    \dist\bigl( z,\fii_\jjj(A) \bigr) &= \dist\bigl(
    \fii_\iii(x),\fii_\jjj(A) \bigr) \ge \as_\iii \dist\bigl(
    x,\overline{\fii}_\iii \bigl( \fii_\jjj(A) \bigr) \bigr) \\
    &\ge \as_\iii \dist\bigl( x,F_A \cup (\R^d \setminus \Omega)
    \bigr) > \as_\iii D\dist(x,A) \\
    &\ge \as_\iii D\ys_\iii^{-1} \dist\bigl( \fii_\iii(x),\fii_\iii(A)
    \bigr) \ge \dist\bigl( z,\fii_\iii(A) \bigr)
  \end{align*}
  using \eqref{eq:congruent_function}. Changing the roles of $\iii$
  and $\jjj$ above, we end up with a contradiction.
  The proof is finished.
\end{proof}

We say that a tractable IFS $\{ \fii_i : i \in I \}$ satisfies the
ball condition if the corresponding CMC $\{ \fii_\iii(A) :
\iii \in I^* \}$ satisfies the (uniform) ball condition.
By Lemma \ref{thm:congruent_ifs}, this defines the ball
condition also for a semiconformal IFS. In this case, we
may choose $A$ to be the invariant set $E$.
Observe that if the IFS is semiconformal and there exists an open set
$W \subset \Omega$ such that for each $r>0$ we have
\begin{equation*}
  \fii_\iii(W) \cap \fii_\jjj(W) = \emptyset
\end{equation*}
for any two distinct $\iii,\jjj \in Z(r)$ then the ball
condition is satisfied. See Example \ref{ex:lohhari}. In
particular, the OSC implies the ball condition in the semiconformal
case. See also \cite[Proposition 3.6]{Kaenmaki2004}. The following
theorem says that, in fact, the ball condition and the strong OSC are
equivalent. Example \ref{ex:suorakaiteet3} shows that this is not true
for tractable IFS's.

\begin{theorem} \label{thm:IFS_separation_condition}
  A semiconformal IFS satisfies the ball condition exactly
  when $O_E \cap E \ne \emptyset$.
\end{theorem}

\begin{proof}
  Let us first prove that the ball condition implies $O_E \cap
  E \ne \emptyset$. Recall that $X$ 
  is the closed $\eps$-neighborhood of $E$. We may further assume that
  \begin{equation*}
    F := \bigcup_{\iii \bot \jjj} \fii_\iii^{-1}\bigl( \fii_\jjj(E) \bigr)
    \cap X \ne \emptyset
  \end{equation*}
  seeing that $F = \emptyset$ implies $\dist(E,F_E) \ge \eps$, which
  gives $E \subset O_E$. It is now sufficient to find a point $x \in
  E$ with $\dist(x,F) > 0$.

  According to Theorem
  \ref{thm:bounded_separation} and Corollary \ref{thm:schief_cor},
  there exist a point $x \in E$ and a constant $\delta > 0$ such that
  \begin{equation*}
    |\fii_\iii(x) - \fii_{\jjj\hhh}(x)| > \delta\diam\bigl(
    \fii_\iii(X) \bigr)
  \end{equation*}
  whenever $\iii \bot \jjj$ and $\hhh \in I^*$. It is easy to see that
  the set $\{ \fii_{\jjj\hhh}(x) : \hhh \in I^* \}$ is dense in
  $\fii_\jjj(E)$. So, in fact, we have
  \begin{equation*}
    \dist\bigl( \fii_\iii(x), \bigcup_{\iii\bot\jjj} \fii_\jjj(E)
    \bigr) \ge \delta \diam\bigl( \fii_\iii(X) \bigr)
  \end{equation*}
  for each $\iii \in I^*$, which in turn implies that
  \begin{equation*}
    |\fii_\iii(x) - \fii_\iii(y)| \ge \delta \diam\bigl( \fii_\iii(X)
    \bigr) 
  \end{equation*}
  for each $y \in \fii_\iii^{-1}\bigl( \fii_\jjj(E) \bigr)$ when
  $\iii\bot\jjj$. On the other hand, Lemma
  \ref{thm:congruent_ifs} shows that there is a constant
  $C>0$ such that
  \begin{equation*}
    |\fii_\iii(x) - \fii_\iii(y)| \le C\diam\bigl( \fii_\iii(X) \bigr)
    |x-y|
  \end{equation*}
  whenever $x,y \in X$ and $\iii \in I^*$. Combining the inequalities
  above gives
  \begin{equation*}
    |x-y| \ge C^{-1}\delta
  \end{equation*}
  for each $y \in F$ and consequently $\dist(x,F)>0$ as desired.

  Since the other direction follows immediately from Proposition
  \ref{thm:OA_feasible}, the proof is finished.
\end{proof}

The following proposition generalizes \cite[Corollary 2.3]{Schief1994}
and \cite[Corollary 1.2]{PeresRamsSimonSolomyak2001} into the setting
of semiconformal IFS's. Although the argument used here is similar to the
proof of \cite[Corollary 1.2]{PeresRamsSimonSolomyak2001}, we give the
details for the convenience of the reader.

\begin{proposition} \label{thm:closure_interior}
  If a semiconformal IFS satisfies the OSC and $\dimh(E) = d$ then the
  invariant set $E$ is the closure of its interior.
\end{proposition}

\begin{proof}
  As the OSC implies the uniform finite clustering property, we have
  $P(d)=0$. Hence there exists a constant $c>0$ such that
  \begin{equation}
    \label{eq:closint1}
    \sum_{\iii \in I^n} \as_\iii^d \ge D^{-d}\diam(X)^{-d} \sum_{\iii
      \in I^n} \diam\bigl( \fii_\iii(X) \bigr)^d \ge c,
  \end{equation}
  see the defining equation \eqref{eq:topo2} and Lemma
  \ref{thm:summa_arvio}. Choose the forwards invariant feasible set $V
  \subset X$ as in Lemma \ref{thm:classical_OSC} and consider the set
  \begin{equation*}
    T = V \setminus \bigcup_{i \in I} \fii_i(V).
  \end{equation*}
  The facts that $\fii_\iii(T) \subset \fii_\iii(V)$ and $\fii_\iii(T)
  \cap \fii_{\iii\jjj}(V) = \emptyset$ for every $\iii \in I^*$ and
  $\jjj \in I^*$ easily lead to the conclusion that $\fii_\iii(T) \cap
  \fii_\jjj(T) = \emptyset$ whenever $\iii \ne \jjj$. Furthermore,
  since $\fii_\iii(T) \subset X$ for each $\iii \in I^*$, we have
  \begin{equation} \label{eq:closint2}
  \begin{split}
    \infty > \HH^d(X) &\ge \HH^d\biggl( \bigcup_{\iii \in I^*}
    \fii_\iii(T) \biggr)
    = \sum_{n \in \N} \sum_{\iii \in I^n} \HH^d\bigl( \fii_\iii(T)
    \bigr) \\ &\ge \HH^d(T) \sum_{n \in \N} \sum_{\iii \in I^n} \as_\iii^d.
  \end{split}
  \end{equation}
  Now \eqref{eq:closint1} and \eqref{eq:closint2} together imply that
  $\HH^d(T) = 0$. This in turn shows that the set
  \begin{equation*}
    V \setminus \overline{\bigcup_{i \in I} \fii_i(V)} = V \setminus
    \bigcup_{i \in I} \fii_i(\overline{V})
  \end{equation*}
  is empty, being an open set with zero measure. Here with the
  notation $\overline{A}$, we mean the closure of a given set
  $A$. This means that $\overline{V} = \bigcup_{i \in I}
  \fii_i(\overline{V})$, giving $E = \overline{V}$ by the uniqueness
  of the invariant set. The proof is complete.
\end{proof}

A similitude IFS, introduced in \cite{Hutchinson1981}, is the most obvious
example of a semiconformal IFS. Suppose that for each $i \in I$ there is a
mapping $\fii_i \colon \R^d \to \R^d$ and a constant $0<s_i<1$ such that
\begin{equation*}
  |\fii_i(x) - \fii_i(y)| = s_i|x-y|
\end{equation*}
whenever $x,y \in \R^d$. Now for a closed ball $B$ centered at the
origin, we have $\fii_i(B) \subset B$ whenever $i \in I$ provided that
the radius of $B$ is chosen large enough. The collection $\{ \fii_i :
i \in I \}$ is then an IFS and we call it a \emph{similitude IFS}.

The following proposition is a slightly more general result than
\cite[Theorem 1]{BandtHungRao2006}.

\begin{proposition}
  Given a similitude IFS, the set $O_A$ is forwards invariant and
  feasible for the OSC provided that $O_A \ne \emptyset$ and $A \subset
  X$ is forwards invariant.
\end{proposition}

\begin{proof}
  According to Proposition \ref{thm:OA_feasible}, it suffices to show
  that $\fii_i(O_A) \subset O_A$ as $i \in I$. Assume on the
  contrary that there exist $i \in I$ and $x \in O_A$ such that
  $\fii_i(x) \notin O_A$, that is,
  \begin{equation*}
    D\dist\bigl( \fii_i(x),A \bigr) \ge \dist\bigl( \fii_i(x),F_A
    \bigr).
  \end{equation*}
  Notice that here $D$ can be chosen to be one.
  Therefore, since $A \subset \fii_i^{-1}(A)$ and $\fii_i^{-1}(F_A)
  \subset F_A$ for every $i \in I$, we obtain
  \begin{align*}
    \dist(x,F_A) &> D\dist(x,A) \ge D\dist\bigl( x,\fii_i^{-1}(A) \bigr)
    \\ &= s_i^{-1} D\dist\bigl( \fii_i(x),A \bigr) \ge s_i^{-1}
    \dist\bigl( \fii_i(x),F_A \bigr) \\ &= \dist\bigl(
    x,\fii_i^{-1}(F_A) \bigr) \ge \dist(x,F_A).
  \end{align*}
  This contradiction finishes the proof.
\end{proof}

The following corollary summarizes the main implications shown for a
semiconformal IFS. Notice that the topological pressure here is well
defined via Lemma \ref{thm:congruent_ifs} as it does not depend on the
choice of the corresponding forwards invariant set.

\begin{corollary} \label{thm:classical_result}
  For a semiconformal IFS, the following conditions are equivalent:
  \begin{enumerate}
  \item The ball condition.
  \item The open set condition.
  \item The strong open set condition.
  \item $\HH^t(E) > 0$, where $t$ is the zero of the topological
    pressure.
  \end{enumerate}
\end{corollary}

\section{Examples}

In the last section, we illustrate the preceding theory by providing
the reader with several examples. We begin by showing that the uniform
finite clustering property does not imply the bounded overlapping
property.

\begin{example} \label{ex:ekaesim}
  The standard Cantor $\tfrac13$-set $E$ can be defined as the
  invariant set of the similitude IFS formed by the mappings
  \begin{align*}
    \fii_0(x) &= \tfrac13 x, \\
    \fii_1(x) &= \tfrac13 x + \tfrac23
  \end{align*}
  on $\R$. It is well known that $\HH^t(E)=1$, where $t = \log 2/\log
  3$, see \cite[Theorem 1.14]{Falconer1985}.
  Consider now the CMC $\{ \fii_\iii(X) : \iii \in I^* \}$, where $X =
  [0,3]$ and $I=\{ 0,1 \}$. This CMC is tractable by Lemma
  \ref{thm:IFS_controlled_tractable}.  The positivity of $\HH^t(E)$ implies the
  uniform finite clustering property by Theorem \ref{thm:hausdorff_bounded}.
  However, using the facts $1 \in \fii_0(X)$ and $\fii_1(1)=1$,
  we infer by induction that $1 \in \fii_{1^k0}(X)$ for every $k \in \N$,
  where $1^k = (1,\ldots,1) \in I^k$ for each $k$. Since the infinite
  set $\{ 1^k0 : k \in \N \}$ is incomparable, we conclude that the
  bounded overlapping property is not satisfied.
\end{example}

\begin{example} \label{ex:nontractable}
  In this example, we give a CMC which shows that the
  assumption concerning the relative positions of the sets
  $X_\iii$ in the last claim of Proposition \ref{thm:semileikkaus} is
  indispensable. Besides this, it is also an example of a nontractable CMC.
  Using the mappings $\fii_\iii$, $\iii \in I^*$, from the previous
  example, set
  \begin{align*}
    X_0 &= [0,1] \times [0,1], \\
    X_1 &= [0,1] \times [-1,0]
  \end{align*}
  and for $j \in I$ and $\iii \in I^*$
  \begin{equation*}
    X_{j\iii} =
    \begin{cases}
      \fii_\iii([0,1]) \times [0,3^{-|\iii|}], &\text{if } j=0 \\
      \fii_\iii([0,1]) \times [-3^{-|\iii|},0], &\text{if } j=1.
    \end{cases}
  \end{equation*}
  The CMC determined by these squares obviously has the limit $E = E_x
  \times \{ 0 \} \subset \R^2$, where $E_x \subset \R$ is the standard Cantor
  $\tfrac{1}{3}$-set. It is equally obvious that the uniform ball
  condition is satisfied, which, according to Theorems
  \ref{thm:regular} and \ref{thm:bounded_separation} and Remark
  \ref{rem:bounded_hausdorff}, implies that the measure $m$ of
  Proposition \ref{thm:semileikkaus} is proportional to $\HH^t|_E$,
  where $t = \log 2/\log 3$ as in the previous example.
  Consequently, $m(J) > 0$ whenever
  $J$ is one of the line segments $\fii_\iii([0,1]) \times \{ 0 \}$,
  $\iii \in I^*$. Especially,
  \begin{equation*}
    m(X_\iii \cap X_\jjj) > 0
  \end{equation*}
  for incomparable symbols $\iii$ and $\jjj$ satisfying
  $\iii|_1 \ne \jjj|_1$ and 
  $\sigma(\iii) = \sigma(\jjj)$. We have hereby shown that the measure
  $m$ is not $t$-semiconformal. On the other hand, Lemma
  \ref{thm:technical_lemma} implies that the bounded overlapping
  property is satisfied, noting
  that clearly $X_\iii \cap E = \pi([\iii])$ for each $\iii \in
  I^*$. Therefore, the extra assumption in Proposition
  \ref{thm:semileikkaus} is really needed.

  Furthermore, this CMC is not tractable. This can be deduced from the
  fact that
  \begin{equation*}
    \dist(X_{0\iii},X_{1\iii}) = 0
  \end{equation*}
  but
  \begin{equation*}
    \dist(X_{00\iii},X_{01\iii}) \ge \dist(X_{00},X_{01}) = \tfrac{1}{3}
  \end{equation*}
  for every $\iii \in I^*$.
\end{example}


\begin{example} \label{ex:suorakaiteet}
  In this example, we define a class of non-semiconformal tractable
  IFS's. Suppose $I$ is a finite set and for each $i \in I$ there is a
  mapping $\fii_i \colon \R^2 \to \R^2$ such that
  \begin{equation*}
    \fii_i(x,y) = (a_ix+c_i,b_iy+d_i),
  \end{equation*}
  where $0 < b_i < a_i < 1$ and $c_i,d_i \ge 0$.
  Defining $a_\iii = a_{i_1}\cdots a_{i_n}$ and $b_\iii =
  b_{i_1}\cdots b_{i_n}$ for each $\iii = (i_1,\ldots,i_n) \in I^n$
  and $n \in \N$, we have
  \begin{align*}
    \sup_{(x_1,y_1) \ne (x_2,y_2)} \frac{|\fii_\iii(x_1,y_1) -
      \fii_\iii(x_2,y_2)|}{|(x_1,y_1) - (x_2,y_2)|} &= a_\iii, \\
    \inf_{(x_1,y_1) \ne (x_2,y_2)} \frac{|\fii_\iii(x_1,y_1) -
      \fii_\iii(x_2,y_2)|}{|(x_1,y_1) - (x_2,y_2)|} &= b_\iii
  \end{align*}
  as $\iii \in I^*$.
  It is clear that $b_\iii/a_\iii \to 0$ as $|\iii| \to \infty$,
  showing that the IFS $\{ \fii_i : i \in I \}$ is not
  semiconformal. However, by choosing $L = 1 + \max_{i \in
    I}\{c_i,d_i\} / (1 - \max_{i \in I}a_i)$ and $X = [0, L]^2$, we
  get $\fii_i(X) \subset X$ for every $i \in I$ and
  \begin{equation*}
    a_\iii L \le \diam\bigl( \fii_\iii(X) \bigr) \le \sqrt{2}a_\iii L
  \end{equation*}
  for each $\iii \in I^*$. The collection $\{ \fii_\iii(X) :
  \iii \in I^* \}$ is thus a CMC and consequently, the IFS $\{ \fii_i
  : i \in I \}$ is tractable.

  According to Corollary \ref{thm:tractable_corollary}, this
  CMC satisfies the (uniform) ball condition if and only if
  $0 < \HH^t(E) < \infty$, where $E$ is the limit set and $\sum_{i \in
    I} a_i^t = 1$.
  For related dimension results, see \cite{McMullen1984},
  \cite{GatzourasLalley1992}, and \cite{FengWang2005}.

  Observe also that choosing, for example, $I = \{ 0,1 \}$, $0<b_0=b_1
  \le a_0=a_1 \le \tfrac12$, $c_1 > 0$, and $d_0 \ge 0 = c_0 = d_1$,
  it is straightforward to see that the ball condition is
  automatically satisfied.
\end{example}

\begin{example} \label{ex:suorakaiteet3}
  Recall that by Corollary \ref{thm:classical_result}, a semiconformal
  IFS satisfies the
  OSC if and only if it satisfies the ball condition. In this
  example, we show that for a tractable IFS this equivalence is not
  necessarily true.

  In Example \ref{ex:suorakaiteet}, let us choose $I = \{ 0,1 \}$,
  $0<b_0=b_1<a_0=a_1 \le \tfrac12$, $d_0>0$, and $c_0=c_1=d_1=0$. It
  is clear that this tractable IFS satisfies the OSC. It can be seen
  by a straightforward calculation that the uniform finite clustering
  property fails, implying that the uniform ball condition does not
  hold. Alternatively, it follows from the observations done in
  Example \ref{ex:suorakaiteet} that the invariant set $E$ has
  Hausdorff dimension $-\log 2/\log a_0$ provided that the ball
  condition is satisfied. However, since $E$ is clearly a subset of
  $\{ 0 \} \times \R$ having Hausdorff dimension $-\log 2/\log b_0$,
  this cannot be the case.

  We do not know if there exists a tractable IFS satisfying the ball
  condition but not the OSC.
%
\end{example}

\begin{example} \label{ex:lohhari}
  At first glance, it seems that for a semiconformal IFS the OSC
  (especially via Lemma \ref{thm:classical_OSC}) is easier to check
  than the ball condition. However, there are cases when it is much
  more convenient to consider the ball condition rather than the OSC. In
  this example, we consider a familiar self-similar set having this
  property.

  We identify $\R^2$ and $\C$ for notational simplicity
  and we set $\eta = \tfrac 12 + \tfrac{i}{2}$. Let $I =
  \{ 0,1 \}$ and let $\fii_0,\fii_1$ be the similitudes 
  given by the equations
  \begin{align*}
    \fii_0(z) &= \eta z, \\
    \fii_1(z) &= \overline{\eta} z + \eta,
  \end{align*}
  where $z \in \C$ and $\overline{\eta} = \tfrac12 - \tfrac{i}{2}$ is
  the complex conjugate of $\eta$. Notice that the
  contraction ratio of both mappings is $\tfrac{1}{\sqrt{2}}$.
  The invariant set $E$ of the IFS $\{ \fii_0,\fii_1 \}$ is the well
  known \emph{L\'evy's dragon}, see \cite{Levy1938}. Since
  $\HH^2(E)>0$, it follows from \cite[Theorem 2.1]{Schief1994} that
  the OSC is satisfied and hence by \cite[Corollary 2.3 and
  its proof]{Schief1994} and Lemma \ref{thm:classical_OSC}, we
  conclude that a nonempty 
  open set $U$ is feasible only if $U \subset E$. Because of the
  intricate structure of the L\'evy's dragon, such an open
  set is, a priori, virtually impossible to find. However, it is
  straightforward to find an open set $W \subset \C$ satisfying 
  \begin{equation*}
    \text{$\fii_\iii(W) \cap \fii_\jjj(W) = \emptyset$ whenever
      $|\iii| = |\jjj|$ and $\iii \ne \jjj$}
  \end{equation*}
  from which the ball condition follows for any corresponding
  semiconformal CMC. This can be done by choosing
  $W$ to be the interior of the right-angled triangle $\triangle =
  \conv\{ 0,1,\eta \}$ and
  looking at the images $\fii_\iii(W)$ as $\iii \in I^n$ with
  fixed $n$. The calculations for the disjointness of the images are
  straightforward since the vertices $\fii_\iii(0)$ and
  $\fii_\iii(1)$ of $\fii_\iii(\triangle)$ belong
  to the point grid $H_n = \{ \eta^n(k+il) : k,l \in \Z \}$ whereas
  $\fii_\iii(\eta) \in H_{n+1}$, and $\fii_\iii(\triangle) =
  \fii_\jjj(\triangle)$ only if $\fii_\iii(0) = \fii_\jjj(0)$ and
  $\fii_\iii(1) = \fii_\jjj(1)$. See
  \cite[p.\ 222]{Edgar1993} for an illustration and apply the
  calculations done in the appendix of \cite{Kigami1995}.
  Notice that the set $W$ above is not feasible since $W \not\subset
  E$.
\end{example}

\begin{example} \label{ex:conformal_IFS}
  In this example, we note that any conformal IFS is semiconformal.
  Suppose $I$ is a finite set and for each $i \in I$ there is a
  contractive $C^{1+\eps}$ conformal mapping $\fii_i \colon \Omega \to
  \Omega$ defined on an open set $\Omega \subset \R^d$. Assuming there
  exists a closed and nonempty $X \subset \Omega$ satisfying
  \begin{equation*}
    \bigcup_{i \in I} \fii_i(X) \subset X,
  \end{equation*}
  the collection $\{ \fii_i : i \in I \}$ is an IFS and we call it a
  \emph{conformal IFS}. We deduce from the well known bounded
  distortion principle that each conformal IFS is semiconformal. See,
  for example, \cite[Remark 2.3]{MauldinUrbanski1996}. Observe that
  the converse does not necessarily hold. For example, the
  semiconformal IFS constructed in \cite[Example 2.1]{Kaenmaki2006} is
  not conformal.
\end{example}

\begin{example} \label{ex:devils_stairs}
  Observe that any IFS conjugated in a bi-Lipschitz way to a conformal
  IFS is semiconformal. Although the bi-Lipschitz conjugacy preserves
  positivity and finiteness of the Hausdorff measure, the following
  example is of special interest as it emphasizes the fact that the use of
  differentiable mappings is not a necessity in order to prove
  Corollary \ref{thm:classical_result}.

  We set $D' \subset [0,1]^2$ to be the graph of a nondecreasing
  continuous function $F \colon [0,1] \to [0,1]$ satisfying $F(0)=0$
  and $F(1)=1$. A well known nondifferentiable example of this kind of
  function is $x \mapsto \HH^t|_E([0,x])$, where $E$ is the
  standard $\tfrac13$-Cantor set and $t = \log 2/\log 3$. In this
  case, the set $D'$ is known as \emph{Devil's stairs}. We set $D = D'
  \cup \{ (x,x) : |x| > 1 \}$, $L = \{ (x,x) : x \in \R \}$, and
  $\proj_L$ to be the orthogonal projection onto $L$. Now the
  mapping $f = (\proj_L|_D)^{-1} \colon L \to D$ is clearly
  bi-Lipschitz and defining a mapping $g \colon \R^2 \to \R^2$ by
  setting $g(x) = x + f\bigl( \proj_L(x) \bigr) - \proj_L(x)$ for each
  $x \in \R^2$, the reader can easily see that also $g$ is bi-Lipschitz.

  Since the line segment $L \cap [0,1]^2$ is clearly the invariant set
  of the similitude IFS $\{ \fii_i : i \in I \}$, where $\fii_i(x,y) =
  (\tfrac{1}{N} x + \tfrac{i-1}{N},\tfrac{1}{N} y +
  \tfrac{i-1}{N})$ and $I = \{ 1,\ldots,N \}$, the set $D' =
  g(L \cap [0,1]^2)$ is the invariant set
  of a semiconformal IFS $\{ g \circ \fii_i \circ g^{-1} : i \in I
  \}$. Here $N \in \N$ is chosen so large that the mappings $g \circ
  \fii_i \circ g^{-1}$ are contractions.

  Devil's stairs also provides the reader with an example of a
  semiconformal IFS which is not conformal, see \cite[Theorem
  2.1]{Kaenmaki2003}.
\end{example}

\begin{example}
  Defining for $A \subset \R^d$, $x \in \R^d$, and $r > 0$
  \begin{align*}
    \por(A,x,r) = \sup\{ \roo \ge 0 : \; &\text{there is } z \in \R^d
    \text{ such that} \\
    &B(z,\roo r) \subset B(x,r) \setminus A \},
  \end{align*}
  we say that a bounded set $A \subset \R^d$ is \emph{uniformly
    porous} if there are $\roo > 0$ and $r_0 > 0$ such that
  $\por(A,x,r) \ge \roo$ for all $x \in A$ and $0<r<r_0$. The notion
  of porosity has arisen from the study of dimensional estimates
  related to the boundary behavior of various mappings.

  Following the proof of \cite[Theorem 4.1]{KaenmakiSuomala2004}, we
  notice that a uniformly porous set is contained in a limit set of a
  semiconformal CMC satisfying the uniform ball condition such that $\dimm(E)
  \le d-c\roo^d$, see Theorem \ref{thm:regular}.
\end{example}

\bibliographystyle{abbrv}
\bibliography{moran.bib}

\end{document}